\documentclass{amsart}
\usepackage[english]{babel}
\usepackage{amsmath}
\usepackage{yhmath}
\usepackage[all]{xy}
\usepackage{xkeyval}
\usepackage{graphicx}
\usepackage{amssymb,amscd,verbatim, amsthm,amsmath,amsgen,xspace}
\usepackage{latexsym}
\usepackage{amsfonts}
\usepackage{color}
\usepackage{ulem}
\usepackage[all]{xy}
\usepackage{latexsym}
\usepackage{exscale}
\usepackage{comment}

\usepackage{ytableau}
\usepackage{color}

\usepackage[all]{xy}

\usepackage{amsbsy}

\def\frak{\mathfrak}

\newtheorem{prop}[equation]{Proposition}
\newtheorem*{prop*}{Proposition}
\newtheorem{thm}[equation]{Theorem}
\newtheorem*{thm*}{Theorem}
\newtheorem{lem}[equation]{Lemma}
\newtheorem*{lem*}{Lemma}

\newtheorem*{kor*}{Corollary}
\newtheorem{cor}[equation]{Corollary}

\newtheorem{ex}[equation]{Example}

\numberwithin{equation}{section}

\usepackage{listings}
\usepackage{xcolor}

\lstdefinelanguage{Sage}{
	language=Python,
	morekeywords={sage, var, solve, matrix, plot, show},
	keywordstyle=\color{blue}\bfseries,
}

\newcommand{\Lat}{\tilde\La}

\newcommand{\frg}{\mathfrak{g}}

\newcommand{\frk}{\mathfrak{k}}

\newcommand{\frp}{\mathfrak{p}}

\newcommand{\frt}{\mathfrak{t}}

\newcommand{\frso}{\mathfrak{so}}

\newcommand{\Cas}{\operatorname{Cas}}

\newcommand{\bbar}{\,|\,}

\def\bbC{\mathbb{C}}

\def\bbZ{\mathbb{Z}}

\let\ccdot\cdot
\def\cdot{\hbox to 2.5pt{\hss$\ccdot$\hss}}

\newcommand{\p}{{\frak p}}

\newcommand{\eq}{\begin{equation}}
	\newcommand{\eeq}{\end{equation}}
\newcommand{\eqn}{\begin{equation*}}
	\newcommand{\bmul}{\begin{multline*}}
		\newcommand{\eemul}{\end{multline*}}
	\newcommand{\eeqn}{\end{equation*}}
\newcommand{\pf}{\begin{proof}}
	\newcommand{\epf}{\end{proof}}

\newcommand{\la}{\lambda}

\renewcommand{\phi}{\varphi}

\newcommand{\La}{\Lambda}
\newcommand{\Lad}{\Lambda^{\dom}}

\newcommand{\eps}{\varepsilon}
\newcommand{\half}{\frac{1}{2}}
\newcommand{\dom}{\operatorname{dom}}

\let\ssize\scriptstyle
\newif\ifFIRST\newdimen\MAXright\MAXright0pt
\def\sdynkin{\bgroup\eightpoint\dynkin}
\def\endsdynkin{\enddynkin\egroup}
\def\dynkin{\bgroup\FIRSTtrue\hskip.5em\setbox1\hbox{$\diagup$}%
	\setbox2\hbox{$\diagdown$}%
	\setbox0\hbox to2\wd1{\hrulefill}%
	\setbox3\hbox{$\bullet$}%
	\setbox4\hbox{$\times$}%
	\setbox7\hbox{$\circ$}
	\def\whiteroot##1{\ifFIRST\setbox5\hbox{$##1$}\ifdim\wd5>1.3em
		\hskip-.5em\hskip.5\wd5\fi\fi\FIRSTfalse
		\hskip-.25em\raise1.5\wd3\hbox to0pt{\hss\hskip.45em$
			\ssize##1$\hss}\copy7\hskip-.25em\setbox6\hbox{$##1$}
		\MAXright\wd6}
	\def\root##1{\ifFIRST\setbox5\hbox{$##1$}\ifdim\wd5>1.3em%
		\hskip-.5em\hskip.5\wd5\fi\fi\FIRSTfalse%
		\hskip-.25em\raise1.5\wd3\hbox to0pt{\hss\hskip.45em$%
			\ssize##1$\hss}\copy3\hskip-.25em\setbox6\hbox{$##1$}%
		\MAXright\wd6}%
	\def\whitedroot##1{\ifFIRST\setbox5\hbox{$##1$}\ifdim\wd5>1.3em
		\hskip-.5em\hskip.5\wd5\fi\fi\FIRSTfalse
		\hskip-.25em\lower1.8\wd3\hbox to0pt{\hss\hskip.45em$
			\ssize##1$\hss}\copy7\hskip-.25em\setbox6\hbox{$##1$}
		\MAXright\wd6}%
	\def\whiterroot##1{\hskip-.25em\copy7\hbox to0pt{\hskip.3em$\ssize##1$\hss}%
		\hskip-.25em\setbox6\hbox{\hskip.6em$##1##1$}%
		\MAXright\wd6}%
	\def\droot##1{\ifFIRST\setbox5\hbox{$##1$}\ifdim\wd5>1.3em%
		\hskip-.5em\hskip.5\wd5\fi\fi\FIRSTfalse%
		\hskip-.25em\lower1.8\wd3\hbox to0pt{\hss\hskip.45em$%
			\ssize##1$\hss}\copy3\hskip-.25em\setbox6\hbox{$##1$}%
		\MAXright\wd6}%
	\def\rroot##1{\hskip-.25em\copy3\hbox to0pt{\hskip.3em$\ssize##1$\hss}%
		\hskip-.25em\setbox6\hbox{\hskip.6em$##1##1$}%
		\MAXright\wd6}%
	\def\norroot##1{\hskip-.36em\copy4\hbox to0pt{\hskip.3em$\ssize##1$\hss}%
		\hskip-.48em\setbox6\hbox{\hskip.6em$##1##1$}%
		\MAXright\wd6}%
	\def\noroot##1{\ifFIRST\setbox5\hbox{$##1$}\ifdim\wd5>1.3em%
		\hskip-.5em\hskip.5\wd5\fi\fi\FIRSTfalse%
		\hskip-.36em\raise1.5\wd3\hbox to0pt{\hss\hskip.6em$%
			\ssize##1$\hss}\copy4\hskip-.38em\setbox6\hbox{$##1$}%
		\MAXright\wd6}%
	\def\nodroot##1{\ifFIRST\setbox5\hbox{$##1$}\ifdim\wd5>1.3em%
		\hskip-.5em\hskip.5\wd5\fi\fi\FIRSTfalse%
		\hskip-.36em\lower1.8\wd3\hbox to0pt{\hss\hskip.6em$%
			\ssize##1$\hss}\copy4\hskip-.38em\setbox6\hbox{$##1$}%
		\MAXright\wd6}%
	\def\nolink{\hskip\wd0}
	\def\link{\raise.22em\copy0}%
	\def\llink##1{\raise.32em\copy0\hskip-\wd0%
		\raise.12em\copy0\hskip-.5\wd0\hbox to0pt{\hss$##1$\hss}\hskip.5\wd0}%
	\def\lllink##1{\raise.22em\copy0\hskip-\wd0\raise.32em\copy0\hskip-\wd0%
		\raise.12em\copy0\hskip-.5\wd0\hbox to0pt{\hss$##1$\hss}\hskip.5\wd0}%
	\def\rootupright##1{\hbox to0pt{\raise.45em\copy1\hskip-.25em\raise1.3\ht1%
			\hbox{\copy3\hskip.3em$\ssize##1$}\hss}%
		\setbox6\hbox{\hskip.6em\copy1\copy1$##1##1$}%
		\ifdim\MAXright<\wd6\MAXright\wd6\fi}%
	\def\whiterootupright##1{\hbox to0pt{\raise.45em\copy1\hskip-.25em\raise1.3\ht1
			\hbox{\copy7\hskip.3em$\ssize##1$}\hss}
		\setbox6\hbox{\hskip.6em\copy1\copy1$##1##1$}
		\ifdim\MAXright<\wd6\MAXright\wd6\fi}
	\def\norootupright##1{\hbox to0pt{\raise.45em\copy1\hskip-.36em\raise1.3\ht1%
			\hbox{\copy4\hskip.3em$\ssize##1$}\hss}%
		\setbox6\hbox{\hskip.6em\copy1\copy1$##1##1$}%
		\ifdim\MAXright<\wd6\MAXright\wd6\fi}%
	\def\rootdownright##1{\hbox to0pt{\raise-.5em\copy2\hskip-.25em\raise-1.35\ht1%
			\hbox{\copy3\hskip.3em$\ssize##1$}\hss}\setbox6%
		\hbox{\hskip.6em\copy2\copy2$##1##1$}%
		\ifdim\MAXright<\wd6\MAXright\wd6\fi}%
	\def\whiterootdownright##1{\hbox to0pt{\raise-.5em\copy2\hskip-.25em\raise-1.35\ht1
			\hbox{\copy7\hskip.3em$\ssize##1$}\hss}\setbox6
		\hbox{\hskip.6em\copy2\copy2$##1##1$}
		\ifdim\MAXright<\wd6\MAXright\wd6\fi}
	\def\rootdown##1{\hbox to0pt{\hskip-.05em\vrule height.25em depth.65em%
			\hskip-.25em\raise-.95em\hbox{\copy3\hskip.3em$\ssize##1$}\hss}%
		\setbox6\hbox{$##1$}%
		\ifdim\MAXright<\wd6\MAXright\wd6\fi}%
	\def\whiterootdown##1{\hbox to0pt{\hskip-.05em\vrule height.25em depth.65em
			\hskip-.25em\raise-.95em\hbox{\copy7\hskip.3em$\ssize##1$}\hss}
		\setbox6\hbox{$##1$}
		\ifdim\MAXright<\wd6\MAXright\wd6\fi}
	\def\dots{\hskip.5em\cdots\hskip.5em}}%
\def\enddynkin{\ifdim\MAXright>1em\hskip.5\MAXright\else\hskip.5em\fi\egroup}%

\begin{document} 
	
	\title[Unitary highest weight modules]{On the classification of unitary highest weight modules in the exceptional cases} 
	\author{Pavle Pand\v zi\'c}
	\address[Pand\v zi\'c]{Department of Mathematics, Faculty of Science, University of Zagreb, Bijeni\v cka 30, 10000 Zagreb, Croatia}
	\email{pandzic@math.hr}
	\author{Ana Prli\'c}
	\address[Prli\'c]{Department of Mathematics, Faculty of Science, University of Zagreb, Bijeni\v cka 30, 10000 Zagreb, Croatia}
	\email{anaprlic@math.hr}
	\author{Gordan Savin}
	\address[Savin]{Department of Mathematics, University of Utah, Salt Lake City, UT 84112, }
	\email{savin@math.utah.edu}
	\author{Vladim\'{\i}r Sou\v cek}
	\address[Sou\v cek]{Matematick\'y \'ustav UK, Sokolovsk\'a 83, 186 75 Praha 8, Czech Republic}
	\email{soucek@karlin.mff.cuni.cz}
	\author{V\'it Tu\v cek}
	\address[Tu\v cek]{Department of Mathematics, Faculty of Science, University of Zagreb, Bijeni\v cka 30, 10000 Zagreb, Croatia}
	\email{vit.tucek@gmail.com}
	\date{}
	\thanks{P.Pand\v zi\'c and A.Prli\'c are supported by the QuantiXLie Center of Excellence, a project co-financed by the Croatian Government and European Union through the European Regional Development Fund - the Competitiveness and Cohesion Operational Programme (grant PK.1.1.02.0004). V. Tu\v cek was supported by the QuantiXLie Center of Excellence and by the grant GX19-28628X.
		V.~Sou\v cek was supported by the grants GX19-28628X and GA24-10887S of GAČR}
	
	\begin{abstract} 
		{In our previous paper \cite{PPST3}, we gave a complete classification of the unitary highest weight modules for the universal covers of the Lie groups $Sp(2n, \mathbb{R}), SO^{*}(2n)$ and $SU(p, q)$,  using the Dirac inequality and the so called PRV product. In this paper, we complete the classification of the unitary highest weight modules for the remaining cases; i.e., universal covers of the Lie groups $SO_{e}(2, n)$, $E_{6(-14)}$ and $E_{7(-25)}$. We also describe unitary highest weight modules with given infinitesimal characters.}
	\end{abstract}

	\maketitle
	
	\section{Introduction}
	Let $G$ denote the universal cover of one of the Lie groups $SO_{e}(2, n)$, $E_{6(-14)}$ and $E_{7(-25)}$, and let $K$ denote the subgroup of elements of $G$ fixed by a fixed Cartan involution of $G$.
	In all these cases $G$ is of Hermitian type, so there exist unitary highest weight Harish-Chandra modules, i.e., unitary 
	Harish-Chandra modules generated by a highest weight vector that is annihilated by the action of all positive root spaces in the 
	complexified Lie algebra $\frg$ of $G$ (\cite{HC1}, \cite{HC2}). Let $\frg_0=\frk_0\oplus\frp_0$ be the Cartan decomposition of the Lie algebra $\frg_0$
	of $G$ corresponding to the Cartan involution. As usual, we drop the subscript 0 to denote the complexifications.  Since in our cases $(G,K)$ is Hermitian, the $K$-module $\frp$ decomposes into two irreducible submodules, $\frp=\frp^+\oplus\frp^-$, and each of $\frp^\pm$ is an abelian subalgebra of $\frg$. Furthermore, $\frg$ and $\frk$ have equal rank, and we choose a Cartan subalgebra $\frt$ of $\frg$ contained in $\frk$. We fix compatible choices of positive roots for $(\frg,\frt)$ and $(\frk,\frt)$, $\Delta_\frg^+\supset\Delta_\frk^+$,  such that $\frp^+$ is contained in the Borel subalgebra defined by $\Delta_\frg^+$.  Let $W$ and $W_{\frk}$ denote the Weyl groups of $\frg$ and $\frk$ respectively.

	If the highest weight module is in addition a $(\frg, K)$--module, then its highest weight must be the highest weight of a $K$--type, hence it is dominant and integral for $\frk$. For each highest weight $\lambda$, there is a unique irreducible highest weight module with highest weight $\lambda$ and it can be constructed as the unique irreducible quotient $L(\lambda)$ of the generalized Verma module $N(\la)$ with highest weight $\lambda$ (see \cite[Section 1]{PPST3} for more details). We will consider only real weights (i.e., contained in the real span of roots) since otherwise, the associated highest-weight module can not be unitary.  In \cite{PPST3} we determined for which $\frk$-dominant integral and real weight $\lambda$ the module $L(\lambda)$ is unitary in cases when $G$ is the universal cover of the Lie groups $Sp(2n, \mathbb{R}), SO^{*}(2n)$ or $SU(p, q)$. This result was already proved in \cite{EHW} and independently in \cite{J}, but our approach was more elementary. In this paper, we will do the same in cases when $G$ is the universal cover of the Lie groups $SO(2, n)$, $E_6$ and $E_7$. In \cite{PPST3} the essential tool was the Dirac inequality of \cite{P2}. We will explain it briefly and refer the reader to a more detailed introduction of \cite{PPST3}.
	
	Let $B$ be the Killing form on $\frg$ and let $C(\frp)$ be the Clifford algebra of $\frp$ with respect to $B$.
	Let $b_i$ be any basis of $\frp$ and let $d_i$ be the dual basis with respect to $B$. The Dirac operator (defined in \cite{P1}) is defined as
	\[
	D=\sum_i b_i\otimes d_i\qquad \in U(\frg)\otimes C(\frp).
	\]
	The square of $D$ is the spin Laplacean (Parthasarathy \cite{P1}):
	\eq
	\label{D squared}  
	D^2=-(\Cas_\frg\otimes 1+\|\rho\|^2)+(\Cas_{\frk_\Delta}+\|\rho_\frk\|^2),
	\eeq
	where $\Cas_\frg$, $\Cas_{\frk_\Delta}$ are the Casimir elements of $U(\frg)$, $U(\frk_\Delta)$, where 
	$\frk_\Delta$ is the diagonal copy of $\frk$ in $U(\frg)\otimes C(\frp)$ defined by 
	\[
	\frk\hookrightarrow\frg\hookrightarrow U(\frg)\quad\text{and}\quad\frk\to\frso(\frp)\hookrightarrow C(\frp), 
	\]
	and $\rho$ and $\rho_\frk$ are the half sums of positive roots for $\frg$ respectively $\frk$.
	
	If $M$ is a $(\frg,K)$-module, then $D$ and $D^2$ act on $M\otimes S$, where $S$ is the spin module for $C(\frp)$. 
	If $M$ is a unitary $(\frg,K)$-module, we can tensor the inner product on $M$ with the inner product on $S$ (See \cite[2.3.9]{HP2} for more details about the inner product on the spin module) to obtain an inner product on $M\otimes S$, with the property that $D$ is self adjoint with respect to this inner product. It follows that $D^2\geq 0$; this is Parthasarathy's Dirac inequality mentioned above.
	
	Dirac inequality can be rewritten in more concrete terms using the formula \eqref{D squared} for $D^2$ and the relation between Casimir actions and infinitesimal characters. As in the introduction of \cite{PPST3} it can be easily shown that if $L(\la)$ is unitary, then the inequality
	\eq
	\label{ehw di}
	\|\mu+\rho\|\geq \|\la+\rho\|,
	\eeq
	must hold for any $K$-type $\mu$ of $L(\la)$. 
	
	\begin{prop}{\cite[Proposition 3.9.]{EHW}} With the notation as above, $L(\la)$ is unitary if and only if the inequality \eqref{ehw di} holds strictly for any $K$-type $\mu \neq \la$ of $L(\la)$.
	\end{prop}
	The problem with this approach is that it is difficult to determine all the $K$--types of $L(\lambda)$. Luckily, it is enough to consider the so-called PRV components of the tensor products of the irreducible finite-dimensional $\frk$-module $F_{\la}$ with highest weight $\la$ and each $K$--type of the module $S(\frp^{-})$, the so-called Schmid modules, which are very well known (see \cite{S}, \cite{PPST3}).  More precisely, the following corollary holds (see \cite[Corollary 2.5]{PPST3} for the proof):
	\begin{cor}
		\label{cor unit nonunit}
		(1) Let $s_0$ be a Schmid module 
		such that for any Schmid module $s$ of strictly lower level than $s_0$, 
		the strict Dirac inequality 
		\eq
		\label{strict di s}
		\|(\la-s)^++\rho\|^2> \|\la+\rho\|^2
		\eeq
		holds. (Here $(\la - s)^+$ denotes the unique $W_{\frk}$--conjugate of $\la - s$). 
		Furthermore, suppose that
		\[
		\|(\la-s_0)^++\rho\|^2< \|\la+\rho\|^2.
		\]
		Then $L(\la)$ is not unitary.
		
		(2) If \eqref{strict di s} 
		holds for all Schmid modules $s$, then $N(\la)$ is irreducible and unitary.
	\end{cor}
	
	In \cite{PPST3}, we constructed all unitary highest weight modules in the discrete part of the classification via the Segal-Shale-Weil representation and its decompositions, tensor products and  PRV products (see \cite[Definition 2.10.]{PPST3}). Unfortunately, in exceptional cases, we cannot construct unitary highest-weight modules in the discrete part via the PRV products. However, we give a shorter proof of \cite[Proposition 12.6]{EHW} in Section \ref{e6:discrete} using the theta correspondence. For the proof that all highest weight modules in the discrete part of the classification for the case $E_7$ are unitary see \cite[Pages 139 and 140]{EHW}).  The situation is similar to the classical case in the sense that for $E_7$ the proof of \cite{EHW} uses the decomposition of the second tensor power of scalar unitarizable module from the discrete part.

	
	In the second part of the paper, we describe unitary highest-weight modules with given infinitesimal characters.
	
	\section{Classification for $\frso(2, 2n-2), n\geq 3$}
	
	In this case $\frg$ is $\frso(2n, \bbC)$ and $\frk$ is $\frso(2n-2, \bbC)$ plus a one-dimensional center. We use standard coordinates and the standard positive root system. The positive compact roots are 
	\[
	\epsilon_i \pm \epsilon_j, \quad 2 \leq i < j \leq n
	\]
	and the positive noncompact roots are
	\[
	\epsilon_1 \pm \epsilon_j, \quad 2 \leq j \leq n.
	\]
	So the half sum of all positive roots is 
	\[
	\rho = (n-1, \ldots, 0)
	\]
	and a weight $\lambda = (\lambda_1, \ldots, \lambda_n)$ is $\frk$-dominant if and only if
	\[
	\lambda_2 \geq \lambda_3 \geq \cdots \geq \lambda_{n-1} \geq |\lambda_n|.
	\]
	Integrality with respect to $\frk$ amounts to $\lambda_i \pm \lambda_j \in \bbZ$ for $2\leq i, j \leq n.$ 
	
	The highest (noncompact) root is $\beta = \epsilon_1 + \epsilon_2.$ There are just two basic Schmid representations, one with highest weight $\beta$ and the other one with highest weight $\beta_2 = 2\epsilon_1.$
	
	The basic Dirac inequality for a PRV component with a Schmid module 
	$s$
	is
	\begin{equation}\label{eq:so_even_basic}
		\| (\lambda - s)^+ + \rho \|^2 \geq \| \lambda + \rho \|^2
	\end{equation}
	This is equivalent to
	\begin{equation}
		2 \langle \gamma \,|\, \lambda + \rho \rangle \leq \|\gamma\|^2
	\end{equation}
	where $\gamma$ is defined by $(\lambda - s)^+ = \lambda - \gamma.$ 
	Since every root has norm squared equal to 2, this is further equivalent to
	\begin{equation}
		\langle \gamma \,|\, \lambda + \rho \rangle \leq 1.
	\end{equation}
	
	The Weyl group of $\frk$ acts on weights by permuting coordinates $\lambda_2, \ldots, \lambda_n$ and changing an even number of their signs. It is because of these sign changes and spinor representations that we have to distinguish several cases. The following lemma was proved in \cite[Lemma 3.1.]{PPST1}.
	
	\begin{lem}
		The basic Dirac inequalitiy for $s = \beta$ is given by
		\begin{equation}\label{eq:so_even_basic_dirac}
			\begin{aligned}
				\lambda_1 &\leq 0  \text{ for } \lambda  =(\lambda_1, 0,\ldots, 0)  \\
				\lambda_1 &\leq 3/2-n  \text{ for } \lambda  =(\lambda_1, 1/2,\ldots, \pm 1/2) \\
				\lambda_1 + \lambda_2 &\leq 2 + p - 2n  \text{ for } \lambda  = (\lambda_1, \lambda_2, \ldots, \lambda_p, \ldots, \lambda_n) \\
				&   \text{ where } 1 \leq \lambda_2 = \cdots = \lambda_p > \lambda_{p+1} \text{ and } 2 \leq p \leq n.
			\end{aligned}
		\end{equation}
	\end{lem}
	
	We will refer to the first case as the {\it scalar} case, the second case is the {\it spinor} case and the remaining one is the {\it general} case. 
	The following theorem gives a complete classification of unitary highest weight modules for $SO(2,2n-2)$.

	\begin{thm} \label{so_even_main}
		Non-scalar modules are unitarizabile if and only if \eqref{eq:so_even_basic_dirac} holds. If the inequality is satisfied strictly, the modules are irreducible Verma modules. In the scalar case, unitarizable highest weight modules are the trivial module, the Wallach module $(2-n, 0, \ldots, 0)$ and irreducible Verma modules $(\lambda_1, 0, \ldots, 0), \lambda_1 < 2 -n.$
	\end{thm}

	\pf In view of Corollary \ref{cor unit nonunit}, the theorem 
	follows from \cite[Theorems 3.5, 3.6 and 3.7]{PPST1}. 
	Namely, the only module which is not in the continuous part of its line, and is not at the end of the continuous part, is the trivial module, which is of course unitary.
	\epf

	\section{Classification for $\frso(2, 2n-1), \, n \geq 2$}
	
	In this case $\frg$ is $\frso(2n+1, \bbC)$ and $\frk$ is $\frso(2n-1, \bbC)$ plus a one-dimensional center. We use standard coordinates and standard positive root system. The positive compact roots are 
	\[
	\{ \epsilon_i \pm \epsilon_j \,|\, 2 \leq i < j \leq n \} \cup \{\epsilon_i \,|\, 2 \leq i \leq n \}
	\]
	and the positive noncompact roots are
	\[
	\{ \epsilon_1 \pm \epsilon_j \,|\, 2 \leq j \leq n\} \cup \{ \epsilon_1\}
	\]
	So the half sum of all positive roots is 
	\[
	\rho = (n-\frac{1}{2}, n-\frac{3}{2}, \ldots, \frac{1}{2})
	\]
	and a weight $\lambda = (\lambda_1, \ldots, \lambda_n)$ is $\frk$-dominant if and only if
	\[
	\lambda_2 \geq \lambda_3 \geq \cdots \lambda_{n} \geq 0.
	\]
	Integrality with respect to $\frk$ amounts to $\lambda_i - \lambda_j \in \bbZ$ and and $2\lambda_i \in \mathbb{N}$ for all $2\leq i, j \leq n.$
	The highest (noncompact) root is $\beta = \epsilon_1 + \epsilon_2.$ There are just two basic Schmid representations and the second one has the highest weight $\beta_2 = 2\epsilon_1.$
	
	The basic Dirac inequality for a PRV component with a Schmid module is
	\begin{equation}\label{eq:so_odd_basic}
		\| (\lambda - s)^+ + \rho \|^2 \geq \| \lambda + \rho \|^2
	\end{equation}
	This is equivalent to
	\begin{equation}
		2 \langle \gamma \,|\, \lambda + \rho \rangle \leq \|\gamma\|^2
	\end{equation}
	where $\gamma$ is defined by $(\lambda - s)^+ = \lambda - \gamma.$
	
	The Weyl group of $\frk$ acts on weights by permuting coordinates $\lambda_2, \ldots, \lambda_n$ and changing  signs on them. The following lemma was proved in \cite[Lemma 3.3.]{PPST1}.
	
	\begin{lem}
		The basic Dirac inequalitiy for $s = \beta$ is given by
		\begin{equation}\label{eq:so_odd_basic_dirac}
			\begin{aligned}
				\lambda_1 &\leq 0 & \text{ for } \lambda=(\lambda_1, 0,\ldots, 0)  \\
				\lambda_1 &\leq 1-n & \text{ for } \lambda=(\lambda_1, 1/2,\ldots, 1/2) \\
				\lambda_1 + \lambda_2 &\leq 1 + p - 2n &\text{ for } \lambda=(\lambda_1, \lambda_2, \ldots, \lambda_p, \ldots, \lambda_n) \\
				&  & \text{ where } 1 \leq \lambda_2 = \cdots = \lambda_p > \lambda_{p+1} \text{ and } 2 \leq p \leq n.
			\end{aligned}
		\end{equation}
	\end{lem}
	
	The following theorem follows from \cite[Theorems 3.8 and 3.9]{PPST1} and from Corollary \ref{cor unit nonunit}.
	
	\begin{thm}\label{so_odd_main}
		Non-scalar modules are unitarizabile if and only if \eqref{eq:so_odd_basic_dirac} holds. If the inequality is satisfied strictly, the modules are irreducible Verma modules. In the scalar case, unitarizable highest weight modules are the trivial module, the Wallach module $(3/2-n, 0, \ldots, 0)$ and irreducible Verma modules $(\lambda_1, 0, \ldots, 0), \lambda_1 < 3/2 -n.$
	\end{thm}
	
	The above theorem gives a complete classification of unitary highest weight modules for $SO(2,2n-1)$. Namely, the only module which is not in the continuous part of its line, and is not at the end of the continuous part, is the trivial module, which is of course unitary.
	
	\section{Classification for $\mathfrak{e}_{6(-14)}$}
	We consider roots as vectors in a six-dimensional subspace of an eight-dimensional space (see \cite{Kn}), precisely we will consider them as $8$--tuples which have the same sixth and seventh coordinate, and the eight coordinate is equal to minus the sixth coordinate. The positive compact roots are 
	$$
	\eps_i \pm \eps_j, \quad 5 \geq i > j.
	$$
	and the positive noncompact roots are 
	$$
	\frac{1}{2} \left( \eps_8 - \eps_7 - \eps_6 + \sum_{i = 1}^{5}(-1)^{n(i)} \eps_i \right) \text{with }  \sum_{i = 1}^{5} n(i) \text{ even}.
	$$
	In this case
	$$
	\rho = (0, 1 , 2, 3, 4, -4, -4, 4).
	$$
	The set of simple roots is
	\begin{align*}
		\Pi & = \{ \alpha_1, \alpha_2, \alpha_3, \alpha_4, \alpha_5, \alpha_6\} \\
		& = \left \{ \frac{1}{2} \left ( \eps_8 - \eps_7 - \eps_6 - \eps_5 - \eps_4 - \eps_3 - \eps_2 + \eps_1 \right ), \eps_2 + \eps_1, \eps_2 - \eps_1, \eps_3 - \eps_2, \eps_4 - \eps_3, \eps_5 - \eps_4 \right\}.
	\end{align*}
	The numbering of simple roots in the Dynkin diagram is given by 
	$$
	\binom{13456}{2}.
	$$ 
	The strongly orthogonal non-compact positive roots are (see \cite{EJ}) 
	\begin{align*}
		\beta_1 & = \binom{12321}{2} = \alpha_1 + 2 \alpha_3 + 3 \alpha_4 + 2 \alpha_5 + \alpha_6 + 2 \alpha_2 = \frac{1}{2} \left( 1, 1, 1, 1, 1, -1, -1, 1 \right), \\
		\beta_2 & = \binom{11111}{0} =  \alpha_1 +  \alpha_3 +  \alpha_4 + \alpha_5 + \alpha_6  = \frac{1}{2} \left( -1, -1, -1, -1, 1, -1, -1, 1 \right).
	\end{align*}
	The highest weight $(\mathfrak{g} ,K)$--modules have highest weights of the form
	\begin{align*}
		\lambda  = (\lambda_1, \lambda_2, \lambda_3, \lambda_4, \lambda_5, \lambda_6, \lambda_6, - \lambda_6), \quad & |\lambda_1| \leq \lambda_2 \leq \lambda_3 \leq \lambda_4 \leq \lambda_5, \\
		& \lambda_i - \lambda_j \in \mathbb{Z}, \ 2 \lambda_i \in \mathbb{Z}, \quad i, j \in \{1,2,3,4,5\}
	\end{align*}
	The basic Schmid $\mathfrak{k}$--modules in $S(\p^{-})$ have lowest weights $-s_i$, $i = 1, 2$, where
	\begin{align*}
		s_1 & = \beta_1 =  \frac{1}{2} \left( 1, 1, 1, 1, 1, -1, -1, 1 \right), \\
		s_2 & = \beta_1 + \beta_2 = \left( 0, 0, 0, 0, 1, -1, -1, 1 \right).
	\end{align*}
	The basic necessary condition for unitarity is the Dirac inequality
	$$
	||(\lambda - s_1)^{+} + \rho||^2 \geq ||\lambda + \rho||^2.
	$$
	
	The following theorem follows from \cite[Theorems 2.1, 2.2, and 2.3]{PPST2} and from Corollary \ref{cor unit nonunit}.
	
	\begin{thm}{\label{e6-classification}}
		If $\lambda = (0, 0, 0, 0, 0, \lambda_6, \lambda_6, - \lambda_6)$, then $L(\lambda)$ is unitarizable if and only if $\lambda_6 \geq 2$ or $\lambda_6 = 0$ (in this case $L(\lambda)$ is trivial). If $\lambda_6 > 2$, then $L(\lambda) = N(\lambda)$.
		
		If $\lambda = (0, 0, 0, 0, \lambda_5, \lambda_6, \lambda_6, - \lambda_6)$, where $\lambda_5 \neq 0$, then $L(\lambda)$ is unitarizable if $3 \lambda_6 - \lambda_5 \geq 14$ and $L(\lambda)$ is not unitarizable if $3 \lambda_6 - \lambda_5 \in \langle - \infty, 14 \rangle \setminus \{ 8\}$. If $3 \lambda_6 - \lambda_5 > 14$, then $L(\lambda) = N(\lambda)$.
		
		If $\lambda$ is such that $(\lambda_1, \lambda_2, \lambda_3, \lambda_4) \neq (0, 0, 0, 0)$, then $L(\lambda)$ is unitarizable if and only if the basic Dirac inequality 
		$$
		||(\lambda - s_1)^{+} + \rho||^2 \geq ||\lambda + \rho||^2
		$$
		holds. If the strict basic Dirac inequality holds, then $L(\lambda) = N(\lambda)$. 
	\end{thm}
	
	The only modules that are not in the continuous part of their lines, and are not at the end of the continuous part, are the trivial module, which
	is of course unitary, and a series of modules of the form $L(0, 0, 0, 0, 3 \lambda_6 - 8, \lambda_6,\la_6, - \lambda_6)$, such that $3\la_6-8$ is a positive integer. These modules are also unitary.  
	This follows from \cite[Corollary 12.6.]{EHW} where the authors verified \eqref{ehw di} directly on all $K$-types. We will give a shorter proof in Section \ref{e6:discrete}. 
	
	\section{Classification for $\mathfrak{e}_{7(-25)}$}
	We consider roots as vectors in a seven-dimensional subspace of an eight-dimensional space (see \cite{Kn}), precisely we will consider them as $8$--tuples which have the eighth coordinate equal to minus the seventh coordinate. The positive compact roots are 
	$$
	\eps_i \pm \eps_j, \quad 5 \geq i > j,
	\quad 
	\frac{1}{2} \left( \eps_8 - \eps_7 - \eps_6 + \sum_{i = 1}^{5}(-1)^{n(i)} \eps_i \right) \text{ with }  \sum_{i = 1}^{5} n(i) \text{ even}.
	$$
	and the positive noncompact roots are 
	$$
	\eps_6 \pm \eps_i, \quad 1 \leq i \leq 5, \quad \eps_8 - \eps_7,
	\quad 
	\frac{1}{2} \left( \eps_8 - \eps_7 + \eps_6 + \sum_{i = 1}^{5}(-1)^{n(i)} \eps_i \right) \text{ with }  \sum_{i = 1}^{5} n(i) \text{ odd}. 
	$$
	In this case
	$$
	\rho = \left( 0, 1 , 2, 3, 4, 5, -\frac{17}{2}, \frac{17}{2} \right ).
	$$
	The set of simple roots is
	\begin{align*}
		\Pi & = \{ \alpha_1, \alpha_2, \alpha_3, \alpha_4, \alpha_5, \alpha_6, \alpha_7 \} \\
		& = \left \{ \frac{1}{2} \left ( \eps_8 - \eps_7 - \eps_6 - \eps_5 - \eps_4 - \eps_3 - \eps_2 + \eps_1 \right ), \eps_2 + \eps_1, \eps_2 - \eps_1, \eps_3 - \eps_2, \eps_4 - \eps_3, \eps_5 - \eps_4, \eps_6 - \eps_5 \right\}.
	\end{align*}
	The numbering of simple roots in the Dynkin diagram is given by 
	$$
	\binom{134567}{2}.
	$$ 
	The strongly orthogonal non-compact positive roots are (see \cite{EJ}) 
	\begin{align*}
		\beta_1 & = \binom{234321}{2} = 2\alpha_1 + 3 \alpha_3 + 4 \alpha_4 + 3 \alpha_5 + 2 \alpha_6 + \alpha_7 + 2 \alpha_2 = \left( 0, 0, 0, 0, 0, 0, -1, 1 \right), \\
		\beta_2 & = \binom{012221}{1} =  \alpha_3 +  2(\alpha_4 + \alpha_5 + \alpha_6) + \alpha_7 + \alpha_2  =  \left( 0, 0, 0, 0, 1, 1, 0, 0 \right), \\
		\beta_3 & = \alpha_7 =   \left( 0, 0, 0, 0, -1, 1, 0, 0 \right).
	\end{align*}
	
	The basic Schmid $\mathfrak{k}$--modules in $S(\p^{-})$ have lowest weights $-s_i$, $i = 1, 2, 3$, where
	\begin{align*}
		s_1 & = \beta_1 =  \left( 0, 0, 0, 0, 0, 0, -1, 1 \right), \\
		s_2 & = \beta_1 + \beta_2 = \left( 0, 0, 0, 0, 1, 1, -1, 1 \right), \\
		s_3 & = \beta_1 + \beta_2  + \beta_3= \left( 0, 0, 0, 0, 0, 2, -1, 1 \right).
	\end{align*}
	The highest weight $(\mathfrak{g} ,K)$--modules have highest weights of the form
	\begin{align*}
		& \lambda  = (\lambda_1, \lambda_2, \lambda_3, \lambda_4, \lambda_5, \lambda_6, \lambda_7, - \lambda_7), \quad  |\lambda_1| \leq \lambda_2 \leq \lambda_3 \leq \lambda_4 \leq \lambda_5, \\
		& \lambda_i - \lambda_j \in \mathbb{Z}, \ 2 \lambda_i \in \mathbb{Z}, \ 1 \leq i \leq j \leq 5\\
		& \frac{1}{2} \left( \lambda_8 - \lambda_7 - \lambda_6 -\lambda_5 -\lambda_4 -\lambda_3 - \lambda_2 + \lambda_1  \right) \in \mathbb{N}_{0}.
	\end{align*}
	The basic necessary condition for unitarity is the Dirac inequality
	$$
	||(\lambda - s_1)^{+} + \rho||^2 \geq ||\lambda + \rho||^2.
	$$
	
	The following theorem follows from \cite[Theorems 2.4, 2.5, and 2.6]{PPST2} and from Corollary \ref{cor unit nonunit}.
	
	\begin{thm}{\label{e7-classification}}
		If $\lambda_ i = 0, \ i \in \{1, 2, 3, 4, 5\}, \lambda_6 = - 2 \lambda_7$, then $L(\lambda)$ is unitarizable if $\lambda_7 \geq 4$ or $\lambda_7 = 0$ (in this last case $L(\lambda)$ is trivial) and $L(\lambda)$ is not unitarizable in case $\lambda_7 \in \langle -\infty, \, 0  \rangle \cup\langle 0, \, 2\rangle \cup \langle 2, \, 4 \rangle$. If $\lambda_7 > 4$, then $L(\lambda) = N(\lambda)$. 
		
		If $\lambda_ i = 0, \ i \in \{1, 2, 3, 4\}, \  \lambda_5 > 0, \ - \lambda_5 - \lambda_6 - 2 \lambda_7 = 0$, then $L(\lambda)$ is unitarizable if $\lambda_7 \geq 6$ and $L(\lambda)$ is not unitarizable if $\lambda_7 \in \langle - \infty, \, 6 \rangle \setminus{4}$. If $\lambda_7 > 6$, then $L(\lambda) = N(\lambda)$. 
		
		In all other cases $L(\lambda)$ is unitarizable if and only if the basic Dirac inequality 
		$$
		||(\lambda - s_1)^{+} + \rho||^2 \geq ||\lambda + \rho||^2
		$$
		holds. If the strict basic Dirac inequality holds, then $L(\lambda) = N(\lambda)$.
	\end{thm}
	
	The only modules that are not in the continuous part of their lines, and are not at the end of the continuous part, are the trivial module, which is of course unitary, the Wallach module $L(0, 0, 0, 0, 0, -4, 2, -2)$ and a series of modules of the form $L(0, 0, 0, 0, \lambda_5, -\lambda_5 - 8, 4, -4)$, such that $\la_5$ is a positive integer.
	These modules are also unitarizable. For a proof, we refer to \cite[Pages 139 and 140]{EHW}.

	\section{An $E_6$ dual pair}\label{e6:discrete}
	\subsection{Minimal representation of $E_7$} 
	
	Let $\mathfrak g$ be a split (complex) Lie algebra of type $E_7$. We fix a root space decomposition of $\mathfrak g$. 
	We identify the root system to the one from Bourbaki tables.
	It sits in the hyperplane in $\mathbb R^8$ perpendicular to $e_7 + e_8$.  Positive roots are 
	\[ 
	\pm e_i + e_j, \, (1\leq i < j \leq 6), \, e_8-e_7 
	\] 
	and  
	\[ 
	\frac{1}{2}(e_8-e_7 +\sum_{i=1}^6 (-1)^{\nu(i)} e_i ) \text{ with  $\sum_{i=1}^6 {\nu(i)}$ odd} 
	\] 
	We identify roots and co-roots so that the natural pairing between root and co-root lattices are given by the usual dot product. 
	Let 
	\[ 
	z=e_6 +\frac12 (e_8-e_7) \in \mathfrak g 
	\] 
	be the 7-th fundamental co-weight for $E_7$. The 7-th fundamental weight is minuscule, thus the adjoint action of $z$ om $\mathfrak g$ 
	gives a short $\mathbb Z$-grading 
	\[  
	\mathfrak g =\bar{ \mathfrak n} \oplus \mathfrak m \oplus \mathfrak n 
	\] 
	where $\mathfrak m$ is the centralizer of $z$ and $\mathfrak n$ is spanned by roots $\alpha$ such that $\alpha(z)= 1$. 
	These roots are 
	\[ 
	\pm e_i + e_6, \,  (1\leq i<6), \,  e_8-e_7 
	\] 
	and  
	\[ 
	\frac{1}{2}(e_8-e_7 + e_6 + \sum_{i=1}^5 (-1)^{\nu(i)} e_i ) \text{ with  $\sum_{i=1}^5 {\nu(i)}$ odd}. 
	\]

	The derived algebra  $\mathfrak m_1= [\mathfrak m, \mathfrak m]$ is of type 
	$E_6$ and $\mathfrak m= \mathfrak m_1+ \mathbb Cz$.  Observe that the centralizer of $\mathfrak m$  in $\mathfrak g$ is $\mathbb C z$, thus 
	$(\mathfrak m, \mathbb Cz)$ is a dual pair in $\mathfrak g$.

	\vskip 5pt 
	
	Let $E_{n}[m]$ denote the irreducible representation of $\mathfrak m$ such that, as a representation of 
	$\mathfrak m_1$, the highest weight is $n\varpi_1$ where 
	\[ 
	\varpi_1 =  \frac{2}{3}(e_8-e_7-e_6), 
	\] 
	and $z$ acts by multiplication by $m$. Observe that $\mathfrak n\cong E_{1}[1]$.  
	The minimal representation is an irreducible, unitarizable $(\mathfrak g, M)$-module with types \cite[Theorem C]{SS}
	\[ 
	\Pi=\oplus_{n=0}^{\infty} E_{n}[n+6]. 
	\] 
	This module itself is a correspondence between $\mathbb Cz$-modules and $\mathfrak m_1$-modules. On the level of infinitesimal characters this is 
	\[ 
	n \leftrightarrow \rho+ (n-6)\varpi_1. 
	\] 
	
	\subsection{A branching}    
	Let 
	\[ 
	h= 2(e_8-e_7-e_6) \in \mathfrak m _{1}
	\] 
	Let $\mathfrak l$ be the centralizer of $h$  in  $\mathfrak m _1$.  
	Then  $\mathfrak l= \mathfrak l_{1}+ \mathbb Ch$ where  $\mathfrak l_1= [\mathfrak l, \mathfrak l]$
	is of type $D_6$ with positive roots 
	\[ 
	\pm e_i + e_j, \, (1\leq i < j \leq 5).  
	\] 
	\vskip 5pt 
	
	Consider $\mathfrak n$ as an $\mathfrak l=\mathfrak l_1 + \mathbb Ch$-module. Using $h$-grading, 
	we have an \underline{ugly} decomposition, 
	\[ 
	\mathfrak n = V_{\lambda}(-2) \oplus V_{\mu}(1) \oplus  \mathbb C(4)
	\] 
	where the action of $h$ is given by the integer in the parenthesis. The summands are, respectively, of dimension 10, 16 and 1,  as one checks that they 
	are spanned by roots 
	\[ 
	\pm e_i + e_6, \, (1\leq i \leq 5), 
	\] 
	\[ 
	\frac{1}{2}(e_8-e_7 + e_6 +\sum_{i=1}^5 (-1)^{\nu(i)} e_i ) \text{ with $\sum_{i=1}^5  \nu(i)$ odd}, 
	\] 
	\[ 
	\text{ and }\,  e_8-e_7.
	\] 
	Here $\lambda$ and $\mu$ are the highest weights of the miniscule $\mathfrak l_1$-modules of dimensions 10 and 16, respectively.

	\subsection{Another dual pair}  
	
	Let $W$ be the Weyl group of $\mathfrak  g$. Recall that $-1\in W$. Let 
	\[ 
	z'= -e_6 +\frac12 (e_8-e_7) \in \mathfrak g. 
	\] 
	It is a Weyl group conjugate of $z$, more precisely, apply $-1$ followed by the reflection about $e_8-e_7$.  
	Thus the adjoint action of $z$ on $\mathfrak g$ gives a short $\mathbb Z$-grading 
	\[  
	\mathfrak g =\bar{ \mathfrak n}' \oplus \mathfrak m' \oplus \mathfrak n' 
	\] 
	where $\mathfrak m'$ is the centralizer of $z'$ and $\mathfrak n'$ is spanned by roots $\alpha$ such that $\alpha(z')= 1$. 
	These roots are 
	\[ 
	\pm e_i - e_6, \,  (1\leq i< 5), \,  e_8-e_7 
	\] 
	and  
	\[ 
	\frac{1}{2}(e_8-e_7 - e_6 + \sum_{i=1}^5 (-1)^{\nu(i)} e_i ) \text{ with  $\sum_{i=1}^5 {\nu(i)}$ even}. 
	\] 
	The derived algebra  $\mathfrak m_1'= [\mathfrak m', \mathfrak m']$ is of type 
	$E_6$,  $\mathfrak m'= \mathfrak m'_1+ \mathbb Cz'$ and 
	$(\mathfrak m', \mathbb Cz')$ is a dual pair in $\mathfrak g$.  
	
	\vskip 5pt 
	Let 
	\[ 
	h'= 2(e_8-e_7+e_6) \in \mathfrak m _{1}'. 
	\] 
	Let $\mathfrak l'$ be the centralizer of $h'$  in  $\mathfrak m'_1$.  
	Then  $\mathfrak l'= \mathfrak l'_{1}+ \mathbb Ch'$ where  $\mathfrak l'_1= [\mathfrak l', \mathfrak l']$. 
	Since $z,h$ and $z',h'$ span the same two-dimensional space, it is clear that $\mathfrak l'_1= \mathfrak l_1$.

	We record some relations for later use. 
	\[ 
	h+h'= 4(z+z')  \text{ and } h'-h=2(z-z')  
	\] 
	\[ 
	h'= 3z+z'  \text{ and } h=3z'+z. 
	\]

	\subsection{An exceptional theta correspondence} 
	
	Write 
	\[ 
	\Pi=\bigoplus_{m\in \mathbb Z} \Pi[m-2]' 
	\] 
	where  $\Pi[m-2]'$ is the summand where $z'$ acts by multiplication by $m-2$. This is clearly a unitarizable $(\mathfrak m_1', \mathfrak l')$-module.   
	Since $z$ has only positive eigenvalues, 
	\[ 
	h'=3z+z' \geq  z', 
	\] 
	thus $\Pi[m-2]'$ are lowest weight modules.  We shall determine their type structure and prove that they are irreducible. 
	
	\vskip 5pt 
	
	Consider $\mathfrak n$ as an $\mathfrak l_1 + \mathbb Cz'$-module.  Since 
	\[ 
	3z'= -z + h
	\]   
	we have a \underline{nice} decomposition, 
	\[ 
	E_1[1]\cong \mathfrak n = V_{\lambda}[-1]' \oplus V_{\mu}[0]' \oplus  \mathbb C[1]'
	\] 
	where the action of $z'$ is given by the integer in the parenthesis. We shall need the following branching rule, Lemma 5.5.1 in   \cite{L}.

	\begin{lem}  As  $\mathfrak l_1 + \mathbb Cz'$-module, $E_{n}[n]$ is a sum 
		\[ 
		V_{a \lambda + b \mu }[-a + c]'
		\] 
		over all non-negative integers $a+b+c =n$.  
	\end{lem} 
	
	Using the lemma, the types of $\Pi[m-2]'$ are easily determined. 
	In view of the relation $3z'=h-z$, when restricting from $E_{n}[n+6]$, the factor $[-a+c]'$ needs to be replaced by $[-a+c-2]'$. Thus 
	\[  
	\Pi[m-2]' = \oplus_{n=0}^{\infty} \oplus_{a+b+c=n, -a+c=m} V_{a \lambda + b \mu }(3n+m+16)' 
	\] 
	where the integer in the last parenthesis denotes the action of $h'=3z+z'$.  (Note that $z'=m-2$ on $\Pi[m-2]'$ and $z=n+6$ on 
	$E_n[n+6]$.) 
	
	\vskip 5pt 
	We now explicate minimal types. There are two cases. Assume $m\geq 0$.  Then the first type of $\Pi[m-2]'$  appears in the 
	restriction of $E_m[m+6]$ and it is for $a=b=0$ and $c=m$. Thus this is the one-dimensional type 
	\[ 
	\mathbb C(4m+16)'. 
	\] 
	Assume $m\leq 0$. Then the first type of $\Pi[m-2]'$ appears in the 
	restriction of $E_{|m|}[|m|+6]$ and it is for $a=|m|$ and $b=c=0$. Thus the minimal type is 
	\[ 
	V_{|m|\lambda} (2|m|+16)'. 
	\] 
	
	\vskip 5pt 
	\begin{prop} $\Pi[m-2]'$ are irreducible. 
	\end{prop}
	\begin{proof} 
		By the matching of infinitesimal characters for the correspondence all submodules of $\Pi[m-2]'$ have the same infinitesimal character. 
		By replacing $h$ by $-h$, we can pass to the language of highest weight modules. 
		Let $\chi$ be the highest weight of  $\Pi[m-2]'$. The infinitesimal character is $\chi+\rho$. 
		Let $F_{\mu}$ be a type in $\Pi[m-2]'$  generating a proper submodule, if any. Then its infinitesimal character is $\mu+\rho$.   Since $\mu$ are known, 
		we can check that $||\chi+\rho|| < ||\mu+\rho||$ for all $\mu\neq \chi$. Hence the submodule would have a different infinitesimal character, a contradiction. 
		
	\end{proof} 
	
	\subsection{Application}  
	Firstly we rephrase our result in terms of $E_6$ centralized by $z$. 
	(We remove $'$ everywhere.) 
	For every  integer $k \geq 0$ we have a unitarizable module $\Pi[-k-2]$ with the lowest $h$-weight $2k+16$ 
	and the corresponding minimal $\mathfrak l_1$-type $V_{k\lambda}$ where $V_{\lambda}$ is the 10-dimensional  representation. 
	Replacing $h$ by $-h$ we have the highest weight $-2k-16$.

	\vskip 5pt 
	
	On the other hand we want to prove unitarity of the module with the highest weight 
	\[ 
	\chi= (0,0,0,0, k, l,l,-l) 
	\] 
	with $3l -k=8$. The minimal $\mathfrak l_1$-type is $V_{k\lambda}$ and the highest $h$-weight is 
	\[ 
	\chi(h)= -2l -2l-2l= -6l= -2(k+8)= -2k-16. 
	\] 
	Thus this module is isomorphic to the unitarizable module $\Pi[-k-2]$. 
	
	\section{Unitary highest weight modules with fixed
		infinitesimal character}
	
	\subsection{The case of $\frso(2, 2n - 2), \, n \geq 3$.} Recall that $\rho=(n-1,n-2,\dots,0)$ and that all $\frk$-dominant weights $\lambda$ are of the form
	$$
	\lambda = (\lambda_1, \lambda_2, \ldots, \lambda_n), 
	$$
	where $\lambda_2 \geq \lambda_3 \ldots \geq \lambda_{n-1} \geq |\lambda_n|$. Since we also assume that $\lambda$ is $\frk$--integral, we have $\lambda_i \pm \lambda_j \in \mathbb{Z}$ for $2 \leq i, j \leq n$,
	or equivalently, $\la_2,\dots,\la_n\in\bbZ$ or $\la_2,\dots,\la_n\in\half+\bbZ$.
	
	By definition, the parameter of the irreducible highest weight $(\frg, K)$--module $L(\lambda)$  is $\Lambda = \lambda + \rho$. It is a particular representative of the infinitesimal character of $L(\lambda)$. We will denote by $\Lambda^{\dom}$ the $\frg$--dominant representative of the infinitesimal character $\Lambda$, i.e., we have
	$$
	\Lambda_{1}^{\dom} \geq \Lambda_{2}^{\dom} \geq \ldots \geq \Lambda_{n-1}^{\dom} \geq |\Lambda_{n}^{\dom}|.
	$$
	All parameters $\Lambda$ corresponding to highest weight $(\frg, K)$--modules must be dominant regular for $\frk$, i.e
	$$
	\Lambda_2 > \Lambda_3 > \ldots > \Lambda_{n-1} > |\Lambda_n|.
	$$
	
	
	Recall that Theorem \ref{so_even_main} implies
	that $L(\lambda)$ is unitary if and only if $\lambda$ is of the form
	\begin{align*}
		& (\lambda_1, 0, \ldots, 0), \quad & \lambda_1  = 0 \text{ or } \lambda_1 \leq 2 - n \\
		& \left ( \lambda_1, \frac{1}{2}, \ldots, \pm \frac{1}{2} \right ), \quad  & \lambda_1  \leq \frac{3}{2} - n
	\end{align*}
	or 
	$$
	(\lambda_1, \lambda_2, \ldots, \lambda_n), \quad  1 \leq \lambda_2 = \cdots = \lambda_p > \lambda_{p + 1}, \quad \, \lambda_1 + \lambda_2 \leq 2 + p - 2n, 
	$$
	for some $p \in \{2, \ldots, n\}$. 
	
	In terms of the corresponding parameter $\Lambda = \lambda + \rho$, $L(\lambda)$ is unitary if and only if $\Lambda$ is of the form
	\begin{align*}
		& (\Lambda_1, n-2, \ldots, 1, 0), \quad & \Lambda_1  = n-1 \text{ or } \Lambda_1 \leq 1 \\
		& \left ( \Lambda_1, n - \frac{3}{2}, \ldots, \frac{3}{2}, \pm \frac{1}{2} \right), \quad  & \Lambda_1  \leq \frac{1}{2}
	\end{align*}
	or
	\begin{align*}
		(\Lambda_1, \Lambda_2, \ldots, \Lambda_n), \quad  & n - 1 \leq \Lambda_2, \quad \Lambda_{i + 1} = \Lambda_{i} - 1, \quad i \in \{2, \ldots, p-1 \}  \\
		& \Lambda_{p + 1} \leq \Lambda_{p} - 2, \quad \Lambda_1 + \Lambda_2 \leq p - 1.
	\end{align*}
	for some $p \in \{2, \ldots, n\}$. 
	
	We want to fix $\Lambda^{\dom}$ and describe a criterion for a parameter $\Lambda$ 
	conjugate to $\Lambda^{\dom}$ to be unitary. 
	
	Note that in order to have any $\frk$-dominant regular integral Weyl group conjugates, $\Lad$ must contain at least $n-1$ coordinates, $\Lad_{i_1},\dots,\Lad_{i_{n-1}}$, so that
	\begin{eqnarray}
		\label{Lad cond}
		&&\text{either } \Lad_{i_1},\dots,\Lad_{i_{n-1}} \in\bbZ \quad\text{or}\quad \Lad_{i_1},\dots,\Lad_{i_{n-1}}\in\half+\bbZ;\\
		\notag &&\Lad_{i_1}>\Lad_{i_2}>\dots>\Lad_{i_{n-2}}>|\Lad_{i_{n-1}}|.
	\end{eqnarray}
	There are now two cases:
	
	{\bf Case 1.} 
	$\Lad$ contains $n-1$ coordinates $\Lad_{i_1},\dots,\Lad_{i_{n-1}}$ satisfying \eqref{Lad cond}, and another coordinate $x$, so that $x$ is either not congruent to $\Lad_{i_1},\dots,\Lad_{i_{n-1}}$ modulo $\bbZ$, or is equal to one of $\Lad_{i_1},\dots,\Lad_{i_{n-2}}, \pm\Lad_{i_{n-1}}$.
	
	{\bf Case 2.} 
	The coordinates of $\Lad$ are either all in $\bbZ$ or all in $\half+\bbZ$, and they satisfy
	\eq
	\label{c4 cond lad}
	\Lad_1>\Lad_2>\dots>\Lad_{n-1}>|\Lad_n|.
	\eeq
	\smallskip
	
	We start by examining Case 1.  
	The assumption implies that the only $\frk$-dominant regular integral conjugates of $\Lad$  are
	\eq
	\label{c3 Las}
	\La=(x\bbar \Lad_{i_1},\dots,\Lad_{i_{n-2}},\Lad_{i_{n-1}})\quad\text{and}\quad \tilde\La=(-x\bbar \Lad_{i_1},\dots,\Lad_{i_{n-2}},-\Lad_{i_{n-1}}).
	\eeq
	The answer to the unitarity question for $\La$ and $\tilde\La$ is the following result.
	
	\begin{thm}
		\label{thm so even c1}
		Assume $\Lad$ is in Case 1 and let $\La,\tilde\La$ be the two possible conjugates of $\Lad$ described in \eqref{c3 Las}.
		Then
		\begin{enumerate}
			\item Suppose that
			\[
			\Lad_{i_1},\dots,\Lad_{i_{n-1}}=n-2,\dots,1,0,
			\]
			so that $\La$ and $\Lat$ are both in the scalar case. 
			Then $\La$ is unitary if and only if $0\leq x\leq 1$, while $\tilde\La$ is unitary if and only if $x\geq 0$.
			\item (a) Suppose that 
			\[
			\Lad_{i_1},\dots,\Lad_{i_{n-1}}=n-\frac{3}{2},\dots,\frac{3}{2},\half,
			\]
			so $\La$ and $\Lat$ are both in the spinor case.
			Then $\La$ is unitary if and only if $-\half\leq x\leq\half$ while $\tilde\La$ is unitary if and only if $x\geq -\half$.
			
			(b) If 
			\[
			\Lad_{i_1},\dots,\Lad_{i_{n-1}}=n-\frac{3}{2},\dots,\frac{3}{2},-\half,
			\]
			so that $\La$ and $\Lat$ are again both in the spinor case, then $\La$ is unitary if and only if $x=\half$, which is already covered in (a), as $\Lat$ for $x=-\half$. Furthermore, $\tilde\La$ is unitary if and only if $x\geq \half$; here the case $x=\half$ is already covered in (a), as $\La$ for $x=-\half$.
			\item Suppose that
			$\Lad_{i_1},\dots,\Lad_{i_{n-1}}$ are neither $n-2,\dots,1,0$ nor $n-\frac{3}{2},\dots,\frac{3}{2},\pm\half$, so that $\La$ and $\Lat$ are both in the general case. 
			Let $p,\tilde p\in\{2,\dots,n\}$ be the integers
			corresponding to $\La,\tilde\La$ (or rather to $\la=\La-\rho$, $\tilde\la=\Lat-\rho$) as in Theorem \ref{so_even_main}.
			
			If $x>0$, then $\La$ is not unitary, while the unitarity of $\tilde\La$ is equivalent to $-x+\Lad_{i_1}\leq \tilde p-1$.
			
			If $x<0$, then $\tilde\La$ is not unitary, while the unitarity of $\La$ is equivalent to $x+\Lad_{i_1}\leq p-1$.
			
			If $x=0$, then $\La$ and $\tilde\La$ are both nonunitary. 
		\end{enumerate}
	\end{thm}
	
	\pf
	(1) Note first that $x$ must be $\geq 0$. 
	By the unitarity criterion of Theorem \ref{so_even_main}, $\La$ is unitary if and only $x\leq 1$, so that $0\leq x\leq 1$ ($x$ can not be $n-1$ since we are assuming $\Lad$ is in Case 1). On the other hand, unitarity of $\tilde\La$ is equivalent to $-x\leq 1$, which is automatic since $x\geq 0$. This proves (1).
	
	(2) (a) 
	Note first that $\Lad_{i_{n-1}}=\half$ implies $x\geq -\half$. By the unitarity criterion of Theorem \ref{so_even_main}, $\La$ is unitary if and only if $x\leq \half$; so $-\half\leq x\leq\half$.
	On the other hand, unitarity of $\tilde\La$ is equivalent to $-x\leq \half$, which is automatic since $x\geq -\half$.
	
	(b) In this case $\Lad_{i_{n-1}}=-\half$, which implies $x\geq \half$. By the unitarity criterion of Theorem \ref{so_even_main}, $\La$ is unitary if and only if $x\leq\half$, which is only possible if $x=\half$.  On the other hand, unitarity of $\tilde\La$ is equivalent to $-x\leq \half$, which is automatic since $x\geq \half$. This proves (2).
	
	(3) The assumptions imply that 
	\begin{eqnarray}
		\label{c1 ladi1 cond}
		&&\Lad_{i_1}\geq n-1 \quad \text{ if }\ \Lad_{i_1},\dots,\Lad_{i_{n-1}}\in\bbZ,\qquad\text{and}\\ \notag
		&& \Lad_{i_1}\geq n-\half\quad \text{ if }\  \Lad_{i_1},\dots,\Lad_{i_{n-1}}\in\half+\bbZ.
	\end{eqnarray}
	Moreover, since $\Lad$ is $\frg$-dominant, $\Lad_{i_{n-1}}<0$ implies $x>0$, and $\Lad_{i_{n-1}}\geq 0$ implies $x\geq -\Lad_{i_{n-1}}$.
	
	Since $\La$ and $\tilde\La$ are both in the general case, the unitarity conditions are $x+\Lad_{i_1}\leq p-1$ respectively $-x+\Lad_{i_1}\leq \tilde p-1$. This, together with $\Lad_{i_1}\geq n-1$ (which follows from \eqref{c1 ladi1 cond}), and with $p,\tilde p\leq n$, implies that 
	$\La$ is not unitary if $x>0$ and that $\tilde\La$ is not unitary if $x<0$. 
	
	If $x> 0$, then the unitarity of $\tilde\La$ is equivalent to $-x+\Lad_{i_1}\leq \tilde p-1$ by Theorem \ref{so_even_main}. 
	
	If $x< 0$, then the unitarity of 
	$\La$ is equivalent to 
	$x+\Lad_{i_1}\leq p-1$ by Theorem \ref{so_even_main}. 
	
	Finally, suppose that $x=0$. If $\Lad_{i_1},\dots,\Lad_{i_{n-2}},\Lad_{i_{n-1}}$ are in $\bbZ$, then $x$ must be a repeated coordinate, so it follows that $\Lad_{i_{n-1}}=0$. (In particular, $\La=\tilde\La$ and $p=\tilde p$.) 
	
	Since $\Lad_{i_1}\geq n-1$, this forces $p<n$, so $x+\Lad_{i_1}=\Lad_{i_1}\geq n-1>p-1$ and $\La=\tilde\La$ is not unitary. If $\Lad_{i_1},\dots,\Lad_{i_{n-2}},\Lad_{i_{n-1}}$ are in $\half+\bbZ$, then it follows from \eqref{c1 ladi1 cond} that $x+\Lad_{i_1}=\Lad_{i_1}\geq n-\half>n-1$. Since $p$ and $\tilde p$ are both $\leq n$, it follows that $\La$ and $\tilde\La$ are both nonunitary.
	\epf
	
	We now handle the remaining case, Case 2. The result is
	
	\begin{thm}
		\label{thm so even c2}
		Let $\Lad$ be as in Case 2. 
		\begin{enumerate}
			\item Suppose that 
			\[
			\Lad=(\Lad_1,n-2,\dots,1,0),
			\]
			with $\Lad_1\in \bbZ$, $\Lad_1\geq n-1$. If $\Lad_1=n-1$, then the unitary conjugates of $\Lad=(n-1,n-2,\dots,1,0)$ are
			\begin{eqnarray*}
				&&(n-1\bbar,n-2,\dots,1,0),\quad \qquad\text{and}\\
				&& (-n+i\bbar n-1,\dots,\widehat{n-i},\dots,0),\quad i=1,\dots,n
			\end{eqnarray*}
			where the hat denotes that the coordinate is omitted.
			
			If $\Lad_1\geq n$, then the only unitary conjugate of $\Lad$ is 
			\[
			(-\Lad_1\bbar n-2,\dots,1,0).
			\]
			\item 
			Suppose that
			\[
			\Lad=(\Lad_1,n-\frac{3}{2},\dots,\frac{3}{2},\pm\half),
			\]
			with $\Lad_1\in\half+\bbZ$ and $\Lad_1\geq n-\half$. If $\Lad_1\geq n+\half$, then the only unitary parameter conjugate to $\Lad$ is
			\[
			(-\Lad_1\bbar n-\frac{3}{2},\dots,\frac{3}{2},\mp\half).
			\]
			If $\Lad=(n-\half,n-\frac{3}{2},\dots,\frac{3}{2},\half)$, then the unitary parameters conjugate to $\Lad$ are $(-n+\half\bbar n-\frac{3}{2},\dots,\frac{3}{2},-\half)$ and 
			\[
			(-(n-i+\half)\bbar n-\half,\dots,\widehat{n-i+\half},\dots,\frac{3}{2},-\half),\quad i=2,3,\dots,n-1.
			\]
			If $\Lad=(n-\half,n-\frac{3}{2},\dots,\frac{3}{2},-\half)$, then the unitary parameters conjugate to $\Lad$ are $(-n+\half\bbar n-\frac{3}{2},\dots,\frac{3}{2},\half)$,
			\[
			(-(n-i+\half)\bbar n-\half,\dots,\widehat{n-i+\half},\dots,\frac{3}{2},\half),\quad i=2,3,\dots,n-1,
			\]
			and $(-\half\bbar n-\half,n-\frac{3}{2},\dots,\frac{3}{2})$.
			
			\item Suppose that $\Lad$ is not as in (1) or (2), i.e., $\Lad_2,\dots,\Lad_n$ are not $n-2,\dots,1,0$ or $n-\frac{3}{2},\dots,\frac{3}{2},\pm\half$.
			Suppose further that $\Lad$ is of the form
			\eq
			\label{c2 nonspec}
			\Lad=(a,a-1,\dots,a-q+1,\Lad_{q+1},\dots\Lad_n),
			\eeq
			with $|\Lad_{q+1}|\leq a-q-1$, 
			for some $q$ such that $1\leq q\leq n-1$. Then the unitary parameters conjugate to $\Lad$ are
			\[
			(-a+i\bbar a,a-1,\dots,\widehat{a-i},\dots,a-q+1,\Lad_{q+1},\dots,-\Lad_n),\qquad i=0,\dots,q-1.
			\] 
			If $\Lad$ is of the form
			\eq
			\label{c2 spec}
			\Lad=(a,a-1,\dots,a-n+2,\pm(a-n+1))
			\eeq
			(with $a\geq n$ so that $\Lad$ is not as in (1) or (2)), then the 
			unitary parameters conjugate to $\Lad$ are 
			\[
			(-a+i\bbar a,a-1,\dots,\widehat{a-i},\dots,a-n+2,\mp(a-n+1)),\qquad i=0,\dots,n-2,
			\]
			and in case $\Lad=(a,\dots,a-n+2,-(a-n+1))$ also
			$(-(a-n+1)\bbar a,a-1,\dots,a-n+2)$.
		\end{enumerate}
	\end{thm}
	
	\pf (1) The parameter $(\Lad_1\bbar n-2,\dots,1,0)$ is in the scalar case, and by Theorem \ref{so_even_main} it is unitary if $\Lad_1=n-1$ and nonunitary if $\Lad_1\geq n$, while $(-\Lad_1\bbar n-2,\dots,1,0)$ is unitary for any $\Lad_1$. The other possible $\La$ for this $\Lad$ are
	\eq
	\label{c1 Las}
	(n-i\bbar \Lad_1,n-2,\dots,\widehat{n-i},\dots,0) \quad\text{and}\quad (-n+i\bbar \Lad_1,n-2,\dots,\widehat{n-i},\dots,0),
	\eeq
	where $i\in\{2,\dots,n\}$, and the hat denotes that the coordinate is omitted. All of these $\La$ are in the general case.
	
	If $\Lad_1=n-1$, then $p=i$ for both $\La$ in \eqref{c1 Las}. For $\La=(n-i\bbar n-1,n-2,\dots,\widehat{n-i},\dots,0)$, we can assume that $i<n$ (if $i=n$ then the two $\La$ in \eqref{c1 Las} are equal). Then we have
	\[
	\La_1+\La_2=2n-i-1 > i-1=p-1,
	\]
	and the module is nonunitary. For $\La=(-n+i\bbar n-1,n-2,\dots,\widehat{n-i},\dots,0)$, 
	\[
	\La_1+\La_2=i-1=p-1,
	\]
	so the module is unitary. 
	
	If $\Lad_1\geq n$, then $p=2$ for both $\La$ in \eqref{c1 Las}. Then $\Lad_1\pm(n-i)\geq i\geq 2>1=p-1$, so both $\La$ in \eqref{c1 Las} are nonunitary. This proves (1).
	\smallskip
	
	(2) The first two conjugates of $\Lad$ we examine are
	\[
	\La=(\Lad_1\bbar n-\frac{3}{2},\dots,\frac{3}{2},\pm\half)\quad\text{and}\quad \Lat=(-\Lad_1\bbar n-\frac{3}{2},\dots,\frac{3}{2},\mp\half).
	\]
	Since $\La$ and $\Lat$ are both in the spinor case, Theorem \ref{so_even_main} implies that $\La$ is not unitary, while $\Lat$ is unitary.
	
	The next conjugates of $\Lad$ we examine are
	\begin{eqnarray*}
		&&\La=(n-i+\half\bbar \Lad_1,n-\frac{3}{2},\dots,\widehat{n-i+\half},\dots,\frac{3}{2},\pm\half)\qquad\text{and}\\
		&&\Lat=(-(n-i+\half)\bbar \Lad_1,n-\frac{3}{2},\dots,\widehat{n-i+\half},\dots,\frac{3}{2},\mp\half),
	\end{eqnarray*}
	where $i\in\{2,3,\dots,n-1\}$. Now $\La$ and $\Lat$ are both in the general case, and Theorem \ref{so_even_main} says that their unitarity depends on the corresponding integers $p$ and $\tilde p$. Since $n-i+\half$ is taken out of the $(i+1)$st place, 
	$p$ and $\tilde p$ are both $\leq i$. If $\Lad_1\geq n+\half$, it follows that  
	\[
	\pm(n-i+\half)+\Lad_1\geq n+\half-(n-i+\half)=i 
	\]
	which is $>p-1$ and $>\tilde p-1$, so $\La$ and $\Lat$ are both nonunitary. If however $\Lad_1=n-\half$, then $p=\tilde p=i$, and while
	$n-i+\half+\Lad_1$ is still $>p-1$, 
	\[
	-(n-i+\half)+\Lad_1 = n-\half-(n-i+\half)=i-1=p-1
	\]
	so $\Lat$ is unitary.
	
	The remaining conjugates of $\Lad=(\Lad_1,n-\frac{3}{2},\dots,\frac{3}{2},\half)$ are
	\begin{eqnarray*}
		&&\La=(\half\bbar \Lad_1,n-\frac{3}{2},\dots,\frac{5}{2},\frac{3}{2})\qquad\text{and}\qquad 
		\Lat=(-\half\bbar  \Lad_1,n-\frac{3}{2},\dots\frac{5}{2},-\frac{3}{2}).
	\end{eqnarray*}
	They again belong to the general case, and $\La$ is nonunitary since $\half+ \Lad_1 \geq n>p-1$. On the other hand, $\Lat$ has a gap at the last coordinate, so $\tilde p<n$ and so
	\[
	-\half+\Lad_1\geq n-1>\tilde p-1,
	\]
	and we see that $\Lat$ is also nonunitary.
	
	Suppose now that $\Lad=(\Lad_1,n-\frac{3}{2},\dots,\frac{3}{2},-\half)$. Then the remaining conjugates of $\Lad$ are
	\[
	\La=(-\half\bbar \Lad_1,n-\frac{3}{2},\dots,\frac{3}{2})\quad\text{and}\quad \Lat=(\half\bbar \Lad_1,n-\frac{3}{2},\dots,\frac{5}{2},-\frac{3}{2}).
	\]
	They are in the general case, and $\Lat$ is nonunitary because as before $\tilde p\leq n-1$ and $\half+ \Lad_1\geq n$. On the other hand, 
	\eq
	\label{c2 last unit}
	-\half+ \Lad_1\geq n-1\geq p-1
	\eeq
	and the inequality is strict (making $\La$ nonunitary) unless $\Lad_1=n-\half$, in which case \eqref{c2 last unit} is an equality and so $\La$ is unitary. This finishes the proof of (2).
	\smallskip
	
	(3) Note first that the condition \eqref{c4 cond lad} together with the fact that coordinates of $\Lad$ do not include $n-2,\dots,1,0$ or $n-\frac{3}{2},\dots,\frac{3}{2},\pm\half$ implies that 
	\eq
	\label{c4 lad1}
	\Lad_1\geq n.
	\eeq
	This immediately implies that the conjugates $\La=\Lad$ and 
	\[
	\La=(\Lad_j\bbar\Lad_1,\Lad_2,\dots,\widehat{\Lad_j},\dots,\Lad_{n-1},\Lad_n),\qquad 2\leq j\leq n-1,
	\]
	of $\Lad$ are all nonunitary. On the other hand, 
	\[
	\La=(\Lad_n\bbar \Lad_1,\Lad_2,\dots,\Lad_{n-1})
	\]
	is nonunitary if $\Lad_n\geq 0$, since in that case $\Lad_n+ \Lad_1 \geq n>p_\Lambda-1$. If $\Lad_n<0$, then unitarity of $\La$ is equivalent to $\Lad_n+\Lad_1=\Lad_1-|\Lad_n|\leq p_\La-1$, but since by \eqref{c4 cond lad} $\Lad_1-|\Lad_n|\geq n-1$, this is equivalent to $\Lad$ being of the form \eqref{c2 spec}, with last coordinate negative.
	
	It remains to consider conjugates
	\[
	\La^{(j)}=(-\Lad_j\bbar\Lad_1,\Lad_2,\dots\widehat{\Lad_j},\dots,\Lad_{n-1},-\Lad_n),\qquad 2\leq j\leq n-1,
	\]
	and 
	\[
	\La=(-\Lad_n\bbar \Lad_1,\Lad_2,\dots,\Lad_{n-2},-\Lad_{n-1})
	\]
	of $\Lad$. If $\Lad$ is as in \eqref{c2 nonspec}, then for $j\leq q$, setting $i=j-1$, we see that for
	\[
	\La^{(j)}=(-a+i\bbar a,a-1,\dots,\widehat{a-i},\dots,a-q+1,\Lad_{q+1},\dots,\Lad_{n-1},-\Lad_n),
	\]
	$p_{\La^{(j)}}=i+1$, so $\La^{(j)}_1+\La^{(j)}_2=-a+i+a=i=p_{\La^{(j)}}-1$, so $\La^{(j)}$ is unitary. If on the other hand $j>q$, then 
	$p_{\La^{(j)}}=q$ and $\La^{(j)}_1+\La^{(j)}_2=-\Lad_j+a\geq q+1 >p_{\La^{(j)}}-1$, so $\La^{(j)}$ is not unitary.
	
	Finally, suppose that $\Lad$ is as in \eqref{c2 spec}. Then by the same argument as above, all parameters 
	\[
	(-a+i\bbar a,a-1,\dots,\widehat{a-i},\dots,a-n+2,\mp(a-n+1)),\qquad 1\leq i\leq n-2,
	\]
	are unitary. 
	
	If $\Lad=(a,a-1,\dots,a-n+2,-(a-n+1))$, then 
	\[
	\La=(a-n+1\bbar a,\dots,a-n+3,-(a-n+2))
	\]
	is nonunitary since $a-n+1+ a\geq n>p_\La-1$.
	
	If $\Lad=(a,a-1,\dots,a-n+2,a-n+1)$, then 
	\[
	\La=(-(a-n+1)\bbar a,\dots,a-n+3,-(a-n+2))
	\]
	is nonunitary, since $p_\La=n-1$ and $\La_1+\La_2=n-1>n-2=p_\La-1$. 
	\epf
	
	\subsection{The case of $\frso(2, 2n - 1), \, n \geq 2$.} Recall that $\rho=(n-\frac{1}{2},n-\frac{3}{2},\dots,\frac{1}{2})$ and that all $\frk$-dominant weights $\lambda$ are of the form
	$$
	\lambda = (\lambda_1, \lambda_2, \ldots, \lambda_n), 
	$$
	where $\lambda_2 \geq \lambda_3  \geq\ldots \geq \lambda_n \geq 0$. Since we also assume that $\lambda$ is $\frk$--integral, we have $\lambda_i \pm \lambda_j \in \mathbb{Z}$ for $2 \leq i, j \leq n$, and $2 \lambda_i \in \mathbb{Z}$, or equivalently $\lambda_2, \ldots, \lambda_n \in \mathbb{Z}$ or $\lambda_2, \ldots, \lambda_n \in \frac{1}{2} + \mathbb{Z}$.

	By definition, the parameter of the irreducible highest weight $(\frg, K)$--module $L(\lambda)$  is $\Lambda = \lambda + \rho$. It is a particular representative of the infinitesimal character of $L(\lambda)$. We will denote by $\Lambda^{\dom}$ the $\frg$--dominant representative of the infinitesimal character $\Lambda$, i.e., we have
	$$
	\Lambda_{1}^{\dom} \geq \Lambda_{2}^{\dom} \geq \ldots \geq \Lambda_{n}^{\dom} \geq 0.
	$$
	All parameters $\Lambda$ corresponding to highest weight $(\frg, K)$--modules must be dominant regular for $\frk$, i.e
	$$
	\Lambda_2 > \Lambda_3 > \ldots > \Lambda_{n-1} > \Lambda_n > 0.
	$$
	
	Recall that Theorem \ref{so_odd_main} implies
	that $L(\lambda)$ is unitary if and only if $\lambda$ is of the form
	\begin{align*}
		& (\lambda_1, 0, \ldots, 0), \quad & \lambda_1 = 0 \text{ or } \lambda_1 \leq \frac{3}{2} - n \\
		& \left ( \lambda_1, \frac{1}{2}, \ldots, \frac{1}{2} \right ), \quad  & \lambda_1  \leq 1 - n
	\end{align*}
	or 
	$$
	(\lambda_1, \lambda_2, \ldots, \lambda_n), \quad  1 \leq \lambda_2 = \cdots = \lambda_p > \lambda_{p + 1}, \quad \, \lambda_1 + \lambda_2 \leq 1 + p - 2n, 
	$$
	for some $p \in \{2, \ldots, n\}$. 
	
	In terms of the corresponding parameter $\Lambda = \lambda + \rho$, $L(\lambda)$ is unitary if and only if $\Lambda$ is of the form
	\begin{align*}
		& \left (\Lambda_1, n-\frac{3}{2}, \ldots, \frac{3}{2}, \frac{1}{2} \right), \quad & \Lambda_1  = n-\frac{1}{2} \text{ or } \Lambda_1 \leq 1 \\
		&  ( \Lambda_1, n - 1, \ldots, 2, 1 ), \quad  & \Lambda_1  \leq \frac{1}{2}
	\end{align*}
	or
	\begin{align*}
		(\Lambda_1, \Lambda_2, \ldots, \Lambda_n), \quad  & n - \frac{1}{2} \leq \Lambda_2, \quad \Lambda_{i + 1} = \Lambda_{i} - 1, \quad i \in \{2, \ldots, p-1 \}  \\
		& \Lambda_{p + 1} \leq \Lambda_{p} - 2, \quad \Lambda_1 + \Lambda_2 \leq p - 1.
	\end{align*}
	for some $p \in \{2, \ldots, n\}$. 
	
	In order to have any $\frk$-dominant regular integral Weyl group conjugates, $\Lad$ must contain at least $n-1$ coordinates, $\Lad_{i_1},\dots,\Lad_{i_{n-1}}$, so that
	\begin{eqnarray}
		\label{Lad cond_odd}
		&&\text{either } \Lad_{i_1},\dots,\Lad_{i_{n-1}} \in\bbZ \quad\text{or}\quad \Lad_{i_1},\dots,\Lad_{i_{n-1}}\in\half+\bbZ;\\
		\notag &&\Lad_{i_1}>\Lad_{i_2}>\dots>\Lad_{i_{n-2}}>\Lad_{i_{n-1}}>0.
	\end{eqnarray}
	There are now two cases:
	
	{\bf Case 1.} 
	$\Lad$ contains $n-1$ coordinates $\Lad_{i_1},\dots,\Lad_{i_{n-1}}$ satisfying \eqref{Lad cond_odd}, and another coordinate $x$ (which must be $\geq 0$), so that $x$ is either not congruent to $\Lad_{i_1},\dots,\Lad_{i_{n-1}}$ modulo $\bbZ$, or is equal to one of $\Lad_{i_1},\dots,\Lad_{i_{n-2}}, \Lad_{i_{n-1}}$ or $x = 0$.
	
	{\bf Case 2.} 
	The coordinates of $\Lad$ are either all in $\bbZ$ or all in $\half+\bbZ$, and they satisfy
	\eq
	\label{c4 cond lad_odd}
	\Lad_1>\Lad_2>\dots>\Lad_{n-1}>\Lad_n > 0.
	\eeq
	\smallskip
	
	We start by examining Case 1.  
	The assumption implies that the only $\frk$-dominant regular integral conjugates of $\Lad$  are
	\eq
	\label{c3 Las_odd}
	\La=(x\bbar \Lad_{i_1},\dots,\Lad_{i_{n-2}},\Lad_{i_{n-1}})\quad\text{and}\quad \tilde\La=(-x\bbar \Lad_{i_1},\dots,\Lad_{i_{n-2}},\Lad_{i_{n-1}}).
	\eeq
	The answer to the unitarity question for $\La$ and $\tilde\La$ is the following result.
	
	\begin{thm}
		\label{thm so odd c1}
		Assume $\Lad$ is in Case 1 and let $\La,\tilde\La$ be the two possible conjugates of $\Lad$ described in \eqref{c3 Las_odd}.
		Then
		\begin{enumerate}
			\item Suppose that
			\[
			\Lad_{i_1},\dots,\Lad_{i_{n-1}}=n-\frac{3}{2},\dots,\frac{3}{2},\frac{1}{2},
			\]
			so that $\La$ and $\Lat$ are both in the scalar case. 
			Then $\La$ is unitary if and only if $x \leq 1$, while $\tilde\La$ is always unitary.
			\item Suppose that 
			\[
			\Lad_{i_1},\dots,\Lad_{i_{n-1}}=n-1,\dots,2,1,
			\]
			so $\La$ and $\Lat$ are both in the spinor case. Then $\La$ is unitary if and only if $x \leq \frac{1}{2}$, while $\tilde\La$ is always unitary.
			\item Suppose that
			$\Lad_{i_1},\dots,\Lad_{i_{n-1}}$ are neither $n-\frac{3}{2},\dots,\frac{3}{2},\half$ nor $n-1,\dots,2,1$, so that $\La$ and $\Lat$ are both in the general case. 
			Let $p,\tilde p\in\{2,\dots,n\}$ be the integers
			corresponding to $\La,\tilde\La$ (or rather to $\la=\La-\rho$, $\tilde\la=\Lat-\rho$) as in Theorem \ref{so_odd_main}. Then $\La$ is never unitary.
			
			If $x \leq \frac{1}{2}$ then $\tilde\La$ is not unitary.
			
			If $x > \frac{1}{2}$, then the unitarity of $\tilde\La$ is equivalent to $-x+\Lad_{i_1}\leq \tilde p-1$.
		\end{enumerate}
	\end{thm}
	
	\pf
	(1)  
	By the unitarity criterion of Theorem \ref{so_odd_main}, $\La$ is unitary if and only $x\leq 1$, so that $0\leq x\leq 1$ ($x$ can not be $n-\frac{1}{2}$ since we are assuming $\Lad$ is in Case 1). On the other hand, unitarity of $\tilde\La$ is equivalent to $-x\leq 1$, which is automatic since $x\geq 0$. This proves (1).
	
	(2) 
	Note first that $x$ must be $\geq 0$. By the unitarity criterion of Theorem \ref{so_odd_main}, $\La$ is unitary if and only if $x\leq \half$; so $0 \leq x\leq\half$.
	On the other hand, unitarity of $\tilde\La$ is equivalent to $-x\leq \half$, which is automatic since $x\geq 0$.
	
	(3) The assumptions imply that 
	\begin{eqnarray}
		\label{c1 ladi1 cond_odd}
		&&\Lad_{i_1}\geq n \quad \text{ if }\ \Lad_{i_1},\dots,\Lad_{i_{n-1}}\in\bbZ,\qquad\text{and}\\ \notag
		&& \Lad_{i_1}\geq n-\half\quad \text{ if }\  \Lad_{i_1},\dots,\Lad_{i_{n-1}}\in\half+\bbZ.
	\end{eqnarray}
	Since $\La$ and $\tilde\La$ are both in the general case, the unitarity conditions are $x+\Lad_{i_1}\leq p-1$ respectively $-x+\Lad_{i_1}\leq \tilde p-1$. This, together with $\Lad_{i_1}\geq n-\frac{1}{2}$ (which follows from \eqref{c1 ladi1 cond_odd}), and with $p,\tilde p\leq n$, implies that 
	$\La$ is never unitary since $x \geq 0$, and  $\tilde\La$ is not unitary if $x< \frac{1}{2}$.  
	
	If $x =\frac{1}{2}$, then the unitarity of 
	$\tilde \La$ is equivalent to 
	$-\frac{1}{2} +\Lad_{i_1}\leq \tilde p-1$ by Theorem \ref{so_odd_main}, i.e. $\Lad_{i_1}\leq \tilde p - \frac{1}{2}$.
	If $\Lad_{i_1},\dots,\Lad_{i_{n-2}},\Lad_{i_{n-1}}$ are in $\bbZ$, then $\Lad_{i_1} \geq n$, so $\tilde \La$ is not unitary.
	If $\Lad_{i_1},\dots,\Lad_{i_{n-2}},\Lad_{i_{n-1}}$ are in $\half+\bbZ$, then it follows from \eqref{c1 ladi1 cond_odd} that $\Lad_{i_1}\geq n-\half$. So, if $\tilde \Lambda$ is unitary, then $\tilde p = n$ and $\Lad_{i_1} = n - \frac{1}{2}$, but in this case we have 
	$$\Lad_{i_1}, \ldots ,\Lad_{i_{n-1}} = n - \frac{1}{2}, \ldots, \frac{5}{2}, \frac{3}{2}, \, x = \frac{1}{2}$$ and this is not Case 1. Therefore, if $x = \frac{1}{2}$, then $\tilde \La$ is not unitary.
	
	If $x > \frac{1}{2}$, then the unitarity of $\tilde\La$ is equivalent to $-x+\Lad_{i_1}\leq \tilde p-1$ by Theorem \ref{so_odd_main}.
	
	\epf
	We now handle the remaining case, Case 2. The result is
	\begin{thm}
		\label{thm so even c2}
		Let $\Lad$ be as in Case 2. 
		\begin{enumerate}
			\item Suppose that 
			\[
			\Lad=(\Lad_1,n-\frac{3}{2},\dots,\frac{3}{2}, \frac{1}{2}),
			\]
			with $\Lad_1\in \frac{1}{2} + \bbZ$, $\Lad_1\geq n- \frac{1}{2}$. If $\Lad_1=n-\frac{1}{2}$, then the unitary conjugates of $\Lad=(n-\half,n- \frac{3}{2},\dots,\frac{3}{2},\half)$ are
			\begin{eqnarray*}
				&&(n-\half \bbar,n-\frac{3}{2},\dots,\frac{3}{2},\half),\quad \qquad\text{and}\\
				&& (-n+i - \half \bbar n-\half ,\dots,\widehat{n-i + \half },\dots,\half),\quad i=1,\dots,n
			\end{eqnarray*}
			where the hat denotes that the coordinate is omitted.
			
			If $\Lad_1\geq n + \half$, then the only unitary conjugate of $\Lad$ is 
			\[
			(-\Lad_1\bbar n-\frac{3}{2},\dots,\frac{3}{2},\half).
			\]
			\item 
			Suppose that
			\[
			\Lad=(\Lad_1,n-1,\dots,2,1),
			\]
			with $\Lad_1\in \bbZ$ and $\Lad_1 \geq n$. If $\Lad_1\geq n+1$, then the only unitary parameter conjugate to $\Lad$ is
			\[
			(-\Lad_1\bbar n-1,\ldots,2,1).
			\]
			If $\Lad=(n,n-1,\dots,2,1)$, then the unitary parameters conjugate to $\Lad$ are 
			\[
			(-n + i \bbar n, n-1, \dots,\widehat{n-i},\dots,2,1),\quad i=0,1,\dots,n-1.
			\]
			
			\item Suppose that $\Lad$ is not as in (1) or (2), i.e., $\Lad_2,\dots,\Lad_n$ are not $n-\frac{3}{2},\dots,\frac{3}{2},\half$ or $n-1,\dots,2,1$.
			Suppose further that $\Lad$ is of the form
			\eq
			\label{c2 nonspec_odd}
			\Lad=(a,a-1,\dots,a-q+1,\Lad_{q+1},\dots\Lad_n),
			\eeq
			with $\Lad_{q+1}\leq a-q-1$, 
			for some $q$ such that $1\leq q\leq n$. Then the unitary parameters conjugate to $\Lad$ are
			\[
			(-a+i\bbar a,a-1,\dots,\widehat{a-i},\dots,a-q+1,\Lad_{q+1},\dots,-\Lad_n),\qquad i=0,\dots,q-1.
			\] 
		\end{enumerate}
	\end{thm}

	\pf (1) The parameter $(\Lad_1\bbar n-\frac{3}{2},\dots,\frac{3}{2},\half)$ is in the scalar case, and by Theorem \ref{so_odd_main} it is unitary if $\Lad_1=n-\half$ and nonunitary if $\Lad_1\geq n + \half$, while $(-\Lad_1\bbar n-\frac{3}{2},\dots,\frac{3}{2},\half)$ is unitary for any $\Lad_1$. The other possible $\La$ for this $\Lad$ are
	\begin{eqnarray}
		\label{c1 Las_odd}
		& (n-i+\half \bbar \Lad_1,n-\frac{3}{2},\dots,\widehat{n-i + \half},\dots,\half) \quad\text{and} \\
		& (-n+ i - \half\bbar \Lad_1,n-\frac{3}{2},\dots,\widehat{n-i + \half},\dots,\half), \notag
	\end{eqnarray}
	where $i\in\{2,\dots,n\}$, and the hat denotes that the coordinate is omitted. All of these $\La$ are in the general case.
	
	If $\Lad_1=n-\half$, then $p=i$ for both $\La$ in \eqref{c1 Las_odd}. For $\La=(n-i + \half \bbar n-\half,n-\frac{3}{2},\dots,\widehat{n-i + \half},\dots,\half)$ we have
	\[
	\La_1+\La_2=2n-i > i-1=p-1,
	\]
	and the module is nonunitary. For $\La=(-n+i - \half \bbar n-\half,n-\frac{3}{2},\dots,\widehat{n-i + \half},\dots,\half)$, 
	\[
	\La_1+\La_2=i-1=p-1,
	\]
	so the module is unitary. 
	
	If $\Lad_1\geq n + \half$, then $p=2$ for both $\La$ in \eqref{c1 Las_odd}. Then $\Lad_1\pm(n-i + \half)\geq i\geq 2>1=p-1$, so both $\La$ in \eqref{c1 Las_odd} are nonunitary. This proves (1).
	\smallskip
	
	(2) The first two conjugates of $\Lad$ we examine are
	\[
	\La=(\Lad_1\bbar n-1,\dots,2,1)\quad\text{and}\quad \Lat=(-\Lad_1\bbar n-1,\dots,2,1).
	\]
	Since $\La$ and $\Lat$ are both in the spinor case, Theorem \ref{so_odd_main} implies that $\La$ is not unitary, while $\Lat$ is unitary.
	
	The next conjugates of $\Lad$ we examine are
	\begin{eqnarray*}
		&&\La=(n-i \bbar \Lad_1,n-1,\dots,\widehat{n-i},\dots,2,1)\qquad\text{and}\\
		&&\Lat=(-n + i \bbar \Lad_1,n-1,\dots,\widehat{n-i},\dots,2,1),
	\end{eqnarray*}
	where $i\in\{1,2,\dots,n-1\}$. Now $\La$ and $\Lat$ are both in the general case, and Theorem \ref{so_odd_main} says that their unitarity depends on the corresponding integers $p$ and $\tilde p$. Since $n-i$ is taken out of the $(i+2)$nd place, 
	$p$ and $\tilde p$ are both $\leq i + 1$. If $\Lad_1\geq n+1$, it follows that  
	\[
	\pm(n-i)+\Lad_1\geq n+1-(n-i)= i +1 
	\]
	which is $>p-1$ and $>\tilde p-1$, so $\La$ and $\Lat$ are both nonunitary. If however $\Lad_1=n$, then $p=\tilde p=i + 1$, and while
	$n-i+\Lad_1$ is still $>p-1$, 
	\[
	-n+i+\Lad_1 = i = \tilde p-1
	\]
	so $\Lat$ is unitary. This finishes the proof of (2).
	\smallskip
	
	(3) Note first that the condition \eqref{c4 cond lad_odd} together with the fact that coordinates of $\Lad$ do not include $n-\frac{3}{2},\dots,\frac{3}{2},\half$ or $n-1,\dots,2,1$ implies that 
	\eq
	\label{c4 lad1_odd}
	\Lad_1\geq n + \half.
	\eeq
	This immediately implies that the conjugates $\La=\Lad$ and 
	\[
	(\Lad_j\bbar\Lad_1,\Lad_2,\dots,\widehat{\Lad_j},\dots,\Lad_{n-1},\Lad_n),\qquad 1 < j\leq n,
	\]
	of $\Lad$ are all nonunitary. 
	It remains to consider conjugates
	\[
	\La^{(j)}=(-\Lad_j\bbar\Lad_1,\Lad_2,\dots\widehat{\Lad_j},\dots,\Lad_{n-1},\Lad_n),\qquad 1\leq j\leq n,
	\]
	of $\Lad$. If $\Lad$ is as in \eqref{c2 nonspec_odd}, then for $j\leq q$, setting $i=j-1$, we see that for
	\[
	\La^{(j)}=(-a+i\bbar a,a-1,\dots,\widehat{a-i},\dots,a-q+1,\Lad_{q+1},\dots,\Lad_{n-1},\Lad_n),
	\]
	$p_{\La^{(j)}}=i+1$, so $\La^{(j)}_1+\La^{(j)}_2=-a+i+a=i=p_{\La^{(j)}}-1$ for $j \neq 1$, so $\La^{(j)}$ is unitary for $j \in \{2,  \ldots, q\}$. If $j = 1$, we have
	\[
	\La^{(1)}=(-a \bbar a-1,\dots, a-q+1,\Lad_{q+1},\dots,\Lad_{n-1},\Lad_n),
	\]
	$p_{\La^{(1)}}=q$, so $\La^{(1)}_1+\La^{(1)}_2=-a+ a -1 = -1 < p_{\La^{(1)}}-1$, so $\La^{(1)}$ is unitary.
	If on the other hand $j>q$, then 
	$p_{\La^{(j)}}=q$ and $\La^{(j)}_1+\La^{(j)}_2=-\Lad_j+a\geq q+1 >p_{\La^{(j)}}-1$, so $\La^{(j)}$ is not unitary.
	\epf
	
	\subsection{The case of $\mathfrak{e}_6$.}
	In this case $\rho = (0, 1, 2, 3, 4, -4, -4, 4)$ and all $\frk$--dominant weights $\la$ are of the form
	$$
	\la = (\la_1, \la_2, \la_3, \la_4, \la_5 \, | \, \la_6, \la_6, - \la_6), 
	$$
	where $|\la_1| \leq \la_2 \leq \la_3 \leq \la_4 \leq \la_5$. Since we also assume that $\la$ is $\frk$--integral, we have $\la_i \pm \la_j \in \bbZ$ for $1 \leq i, j \leq 5$ or equivalently $\la_1, \ldots, \la_5 \in \bbZ$ or $\la_1, \ldots, \la_5 \in \frac{1}{2} + \bbZ$. As before, $\Lad$ is the $\frg$--dominant representative of the infinitesimal character $\Lambda$, i.e. we have
	$$
	|\Lad_1| \leq \Lad_2 \leq \Lad_3 \leq \Lad_4 \leq \Lad_5
	$$
	and
	$$
	\Lad_1 - \sum_{i = 2}^{5} \Lad_i - 3 \Lad_6 \geq 0.
	$$
	All parameters $\La$ corresponding to highest weight $(\frg, K)$--modules must be dominant regular for $\frk$, i.e.
	$$
	|\La_1| < \La_2 < \La_3 < \La_4 < \La_5.
	$$
	Recall that the generators of $W_{\frg}$ are $s_{\eps_i \pm \eps_j}, \, 5 \geq i > j$ and $s_{\alpha_1}$, where $\alpha_1 = \frac{1}{2}(\eps_1 - \sum_{i = 2}^{7} \eps_i + \eps_8)$. First we will determine which $W_{\frg}$--conjugates of $\Lad$ are $\frk$--dominant. We use \textit{SageMath} \cite{Sage} for this purpose.
	\begin{lstlisting}
sage: R = RootSystem("E6").ambient_space()
sage: rho = R. rho()
sage: W = WeylGroup(['E', 6])
sage: def k_dominant(t):
....:     x1, x2, x3, x4, x5, x6, x7, x8 = t
....:     return abs(x1) <= x2 <= x3 <= x4 <= x5
sage: W_1=[]
sage: for alpha in W:
....:     a = alpha.action(rho)
....:     b=(a[0], a[1], a[2], a[3], a[4], a[5], a[6], a[7])
....:     if k_dominant(b):
....:         W_1.append(alpha)
sage: for alpha in W_1:
....:     print (alpha)
....:     print (" ")
	\end{lstlisting}
	Let us define the function $f : \{ x \in \mathbb{R}^8 \, | \, x_6 = x_7 = - x_8 \} \to \mathbb{R}$ as
	$$
	f(x) = 3 x_6 - \sum_{i = 1}^{5} x_i.
	$$
	
	Recall that Theorem \ref{e6-classification} and \cite[p.~732, 733]{PPST2} (cases 1.1, 1.2, 1.3, 1.4, 1.5 and 1.7) imply that $L(\la)$ is unitary if and only if $\la$ is of the form
	\begin{align*}
		& (0, 0, 0, 0, 0 \, | \, \la_6, \la_6, - \la_6 ), \, f(\la)  = 0 \text{ or } f(\la) \geq 6 \\
		& (0, 0, 0, 0, \la_5 \, | \, \la_6, \la_6, - \la_6 ), \la_5 > 0,  \,  f(\la)  = 8 \text{ or } f(\la) \geq 14 \\
		& (\la_1, \la_2, \la_3, \la_4, \la_5 \, | \, \la_6, \la_6, - \la_6), \, \la_1 + \la_2 \geq 1,  f(\la) \geq 20 \\
		& (\la_1, \la_2, \la_3, \la_4, \la_5 \, | \, \la_6, \la_6, - \la_6), \, \la_2 = - \la_1, \, \la_3 - \la_2 \geq 1, \, f(\la) \geq 18 \\
		& (\la_1, \la_2, \la_3, \la_4, \la_5 \, | \, \la_6, \la_6, - \la_6), \, \la_3 = \la_2 = - \la_1, \la_2 > 0, \, \la_4 - \la_2 \geq 1, \,  f(\la) \geq 16 \\
		& (0, 0, 0, \la_4, \la_5 \, | \, \la_6, \la_6, - \la_6),  \, \la_4 \geq 1, \, f(\la) \geq 14 \\
		& (\la_1, \la_2, \la_3, \la_4, \la_5 \, | \, \la_6, \la_6, - \la_6), \, \la_4 = \la_3 = \la_2 = - \la_1, \, \la_2 > 0, \, \la_5 - \la_2 \geq 1, \,  f(\la) \geq 14 \\
		& (\la_1, \la_2, \la_3, \la_4, \la_5 \, | \, \la_6, \la_6, - \la_6), \, \la_5 = \la_4 = \la_3 = \la_2 = - \la_1, \, \la_2 > 0, \,  f(\la) \geq 12. 
	\end{align*}
	
	In terms of the corresponding parameter $\La = \la + \rho$, $L(\la)$ is unitary if and only if $\La$ is of the form
	\begin{align*}
		\text{Case 1: } & (0, 1, 2, 3, 4 \, | \, \La_6, \La_6, - \La_6 ),   \,   f(\La)  = -22 \, \text{ or }  f(\La) \geq -16 \\
		\text{Case 2: } & (0, 1, 2, 3, \La_5 \, | \, \La_6, \La_6, - \La_6 ), \La_5 > 4,   \,  f(\La)  = -14 \, \text{ or }  f(\La) \geq -8 \\
		\text{Case 3: } & (\La_1, \La_2, \La_3, \La_4, \La_5 \, | \, \La_6, \La_6, - \La_6), \, \La_1 + \La_2 \geq 2, \,  f(\La) \geq -2 \\
		\text{Case 4: } & (\La_1, \La_2, \La_3, \La_4, \La_5 \, | \, \La_6, \La_6, - \La_6), \, \La_2 = - \La_1 + 1, \, \La_3 - \La_2 \geq 2,  \, f(\La) \geq -4 \\
		\text{Case 5: } & (\La_1, \La_2, \La_3, \La_4, \La_5 \, | \, \La_6, \La_6, - \La_6), \, \La_3 = \La_2 + 1 = - \La_1 + 2, \, \La_2 > 1, \\
		&  \La_4 - \La_2 \geq 3, \, f(\La) \geq -6 \\
		\text{Case 6: } & (0, 1, 2, \La_4, \La_5 \, | \, \La_6, \La_6, - \La_6),  \, \La_4 \geq 4, \, f(\La) \geq -8 \\
		\text{Case 7: } & (\La_1, \La_2, \La_3, \La_4, \La_5 \, | \, \La_6, \La_6, - \La_6), \, \La_4  = \La_3 + 1 = \La_2 + 2 = - \La_1 + 3, \, \La_2 > 1, \\ & \La_5 - \La_2  \geq 4,  \, f(\La) \geq -8 \\
		\text{Case 8: } & (\La_1, \La_2, \La_3, \La_4, \La_5 \, | \, \La_6, \La_6, - \La_6), \, \La_5 = \La_4 + 1 = \La_3 + 2 = \La_2 + 3 = - \La_1 + 4, \\
		& \La_2 > 1, \,  f(\La) \geq -10. 
	\end{align*}
	For each $\frg$-dominant infinitesimal character $\Lad$, the following Python code in \textit{SageMath} returns a list of unitary and nonunitary infinitesimal characters which are $W_{\frg}$--conjugates of $\Lad$. For a given infinitesimal character $\La$, the first part of the code checks which of the above eight cases $\La$ belongs to. Next, we define a function $f$ as above and then the function \textit{unitary} which checks if the corresponding $\frk$--dominant $W_{\frg}$--conjugate of $\Lad$ is $\frk$--regular and unitary. Finally, the code returns the lists of unitary and nonunitary parameters. 
	\begin{lstlisting}
sage:  def case_1(t):
....: ....:     x1, x2, x3, x4, x5, x6, x7, x8 = t
....: ....:     return (x1 == 0) and (x2 == 1) and (x3 == 2) and (x4 == 3) and (x5 == 4)
sage:  def case_2(t):
....: ....:     x1, x2, x3, x4, x5, x6, x7, x8 = t
....: ....:     return (x1 == 0) and (x2 == 1) and (x3 == 2) and (x4 == 3) and (x5 > 4)
sage:  def case_3(t):
....: ....:     x1, x2, x3, x4, x5, x6, x7, x8 = t
....: ....:     return (x1 + x2 >= 2)
sage:  def case_4(t):
....: ....:     x1, x2, x3, x4, x5, x6, x7, x8 = t
....: ....:     return (x2 == -x1+1) and (x3 - x2 >= 2)
sage:  def case_5(t):
....: ....:     x1, x2, x3, x4, x5, x6, x7, x8 = t
....: ....:     return (x3 == x2 + 1 == -x1 + 2) and (x2 > 1) and (x4 - x2 >= 3)
sage:  def case_6(t):
....: ....:     x1, x2, x3, x4, x5, x6, x7, x8 = t
....: ....:     return (x1 == 0) and (x2 == 1) and (x3 == 2) and (x4 >= 4)
sage:  def case_7(t):
....: ....:     x1, x2, x3, x4, x5, x6, x7, x8 = t
....: ....:     return (x4 == x3 + 1 == x2 + 2 == -x1 + 3) and (x2 > 1) and (x5 - x2 >= 4)
sage:  def case_8(t):
....: ....:     x1, x2, x3, x4, x5, x6, x7, x8 = t
....: ....:     return (x5 == x4 + 1 == x3 + 2 == x2 + 3 == -x1 + 4) and (x2 > 1)
sage: def f(t):
....:     x1, x2, x3, x4, x5, x6, x7, x8 = t
....:     return 3 * (x6) - x1 - x2 - x3 - x4 - x5
sage: def k_dom_reg(t):
....:      x1, x2, x3, x4, x5, x6, x7, x8 = t
....:      return abs(x1) < x2 < x3 < x4 < x5
sage: def g_dom(t):
....:     x1, x2, x3, x4, x5, x6, x7, x8 = t
....:     return abs(x1) <= x2 <= x3 <= x4 <= x5 and x1 - x2 - x3 - x4 - x5 - 3*x6 >= 0
sage:  def unitary(t):
....:      x1, x2, x3, x4, x5, x6, x7, x8 = t
....:      return ((case_1(t) == true) and (k_dom_reg(t) == true) and (f(t)==-22 or f(t) >= - 16)) or ((case_2(t) == true) and  (k_dom_reg(t) == true) and (f(t)==-14 or f(t) >= - 8)) or ((c
....: ase_3(t) == true) and (k_dom_reg(t) == true)and (f(t) >= - 2)) or ((case_4(t) == true) and (k_dom_reg(t) == true) and (f(t) >= - 4)) or ((case_5(t) == true) and (k_dom_reg(t) == true)
....:  and (f(t) >= - 6)) or ((case_6(t) == true) and (k_dom_reg(t) == true) and (f(t) >= - 8)) or ((case_7(t) == true) and (k_dom_reg(t) == true) and (f(t) >= - 8)) or ((case_8(t) == true)
....:  and (k_dom_reg(t) == true)and (f(t) >= - 10))
sage: def unitary_inf_char_list(t):
....:      x1, x2, x3, x4, x5, x6, x7, x8 = t
....:      v = vector([x1, x2, x3, x4, x5, x6, x7, x8])
....:      unitary_list = []
....:      for alpha in W_1:
....:          if (unitary(alpha.to_matrix() * v) and (k_dom_reg(alpha.to_matrix() * v) == true) and ((alpha.to_matrix() * v) not in unitary_list)):
....:              unitary_list.append(alpha.to_matrix() * v)
....:      if not g_dom(t):
....:          return "The entered parameter is not g-dominant"
....:      else:
....:          return(unitary_list)
sage: def nonunitary_inf_char_list(t):
....:      x1, x2, x3, x4, x5, x6, x7, x8 = t
....:      v = vector([x1, x2, x3, x4, x5, x6, x7, x8])
....:      nonunitary_list = []
....:      for alpha in W_1:
....:          if ((unitary(alpha.to_matrix() * v) == false) and (k_dom_reg(alpha.to_matrix() * v) == true) and ((alpha.to_matrix() * v) not in nonunitary_list)):
....:                           nonunitary_list.append(alpha.to_matrix() * v)
....:      if not g_dom(t):
....:           return "The entered parameter is not g-dominant"
....:      else:
....:           return(nonunitary_list)
	\end{lstlisting}
	\begin{ex}
		We get the following lists if we put $t = \rho = (0, 1, 2, 3, 4, -4, -4, 4)$.
		\begin{lstlisting}
sage: t = (0, 1, 2, 3, 4, -4, -4, 4)
sage: unitary_inf_char_list(t)
[(0, 1, 2, 3, 4, -4, -4, 4),
(-3/2, 5/2, 7/2, 9/2, 11/2, 3/2, 3/2, -3/2),
(0, 1, 2, 3, 8, 0, 0, 0),
(-1, 2, 3, 4, 6, 2, 2, -2),
(0, 1, 2, 5, 6, 2, 2, -2),
(-1/2, 3/2, 5/2, 9/2, 11/2, 5/2, 5/2, -5/2),
(0, 1, 3, 4, 5, 3, 3, -3),
(1/2, 3/2, 5/2, 7/2, 9/2, 7/2, 7/2, -7/2),
(0, 1, 2, 3, 4, 4, 4, -4)]
sage: nonunitary_inf_char_list(t)
[(-1/2, 3/2, 5/2, 7/2, 9/2, -7/2, -7/2, 7/2),
(0, 1, 3, 4, 5, -3, -3, 3),
(1/2, 3/2, 5/2, 9/2, 11/2, -5/2, -5/2, 5/2),
(1, 2, 3, 4, 6, -2, -2, 2),
(3/2, 5/2, 7/2, 9/2, 11/2, -3/2, -3/2, 3/2),
(0, 1, 2, 5, 6, -2, -2, 2),
(1/2, 3/2, 5/2, 9/2, 13/2, -3/2, -3/2, 3/2),
(1, 2, 3, 5, 6, -1, -1, 1),
(0, 1, 3, 4, 7, -1, -1, 1),
(1/2, 3/2, 7/2, 9/2, 13/2, -1/2, -1/2, 1/2),
(0, 1, 4, 5, 6, 0, 0, 0),
(-1/2, 3/2, 5/2, 7/2, 15/2, -1/2, -1/2, 1/2),
(0, 2, 3, 4, 7, 0, 0, 0),
(-1/2, 3/2, 7/2, 9/2, 13/2, 1/2, 1/2, -1/2),
(-1, 2, 3, 5, 6, 1, 1, -1),
(1/2, 3/2, 5/2, 7/2, 15/2, 1/2, 1/2, -1/2),
(0, 1, 3, 4, 7, 1, 1, -1),
(-1/2, 3/2, 5/2, 9/2, 13/2, 3/2, 3/2, -3/2)]
		\end{lstlisting}
	\end{ex}
	
	\subsection{The case of $\mathfrak{e}_7$.}
	In this case $\rho = (0, 1, 2, 3, 4, 5, -\frac{17}{2}, \frac{17}{2})$ and all $\frk$--dominant weights $\la$ are of the form
	$$
	\la = (\la_1, \la_2, \la_3, \la_4, \la_5, \la_6 \, | \, \la_7, - \la_7), 
	$$
	where $|\la_1| \leq \la_2 \leq \la_3 \leq \la_4 \leq \la_5$ and $\la_1 - \sum_{i = 2}^{6} \la_i - 2 \la_7 \geq 0$. Since we also assume that $\la$ is $\frk$--integral, we have $\la_i \pm \la_j \in \bbZ$ for $1 \leq i, j \leq 5$ or equivalently $\la_1, \ldots, \la_5 \in \bbZ$ or $\la_1, \ldots, \la_5 \in \frac{1}{2} + \bbZ$ and $\frac{1}{2}(\la_1 - \sum_{i = 2}^{7} \la_i + \la_8) \in \mathbb{Z}$. As before, $\Lad$ is the $\frg$--dominant representative of the infinitesimal character $\Lambda$, i.e. we have
	$$
	|\Lad_1| \leq \Lad_2 \leq \Lad_3 \leq \Lad_4 \leq \Lad_5 \leq \Lad_6
	$$
	and
	$$
	\Lad_1 - \sum_{i = 2}^{6} \Lad_i - 2 \Lad_7 \geq 0.
	$$
	All parameters $\La$ corresponding to highest weight $(\frg, K)$--modules must be dominant regular for $\frk$, i.e.
	$$
	|\La_1| < \La_2 < \La_3 < \La_4 < \La_5 \text{ and } \la_1 - \sum_{i = 2}^{6} \la_i - 2 \la_7 > 0
	$$
	Recall that the generators of $W_{\frg}$ are $s_{\eps_i \pm \eps_j}, \, 6 \geq i > j$ and $s_{\alpha_1}$, where $\alpha_1 = \frac{1}{2}(\eps_1 - \sum_{i = 2}^{7} \eps_i + \eps_8)$. First we will determine which $W_{\frg}$--conjugates of $\Lad$ are $\frk$--dominant. We use a very similar code in \textit{SageMath} \cite{Sage} as for $e_6$, with the exception of the commands
	\begin{center}
		if len$(W_1)==56$:
		
		break    
	\end{center}
	which speed up the program. It is well known  that $|W_1| = \frac{|W_{\frg}|}{|W_{\frk}|} = 56$.
	\begin{lstlisting}
sage: R = RootSystem("E7").ambient_space()
sage: rho = R. rho()
sage: W = WeylGroup(['E', 7])
sage: def k_dominant(t):
....:     x1, x2, x3, x4, x5, x6, x7, x8 = t
....:     return (abs(x1) <= x2 <= x3 <= x4 <= x5 and  x1 - x2 - x3 - x4 - x5 - x6 - 2*x7 >= 0)
sage: W_1=[]
sage: for alpha in W:
....:     a = alpha.action(rho)
....:     b=(a[0], a[1], a[2], a[3], a[4], a[5], a[6], a[7])
....:     if k_dominant(b):
....:         W_1.append(alpha)
....:     if len(W_1)==56:
....:         break
sage: for alpha in W_1:
....:     print (alpha)
....:     print (" ")
	\end{lstlisting}
	Let us define functions $f,g : \{ x \in \mathbb{R}^8 \, | \, x_7 = - x_8 \} \to \mathbb{R}$ by
	$$
	f(x) = x_7, \quad g(x) = \frac{1}{2}(x_1 - \sum_{i = 2}^{6} x_i - 2 x_7)
	$$
	Recall that Theorem \ref{e7-classification} and \cite[p.~743, 744, 745, 746]{PPST2} (cases 1.1, 1.2.1, 1.2.2., 1.2.3, 1.2.4, 1.2.5 and 1.2.7) imply that $L(\la)$ is unitary if and only if $\la$ is of the form
	\begin{align*}
		& (\la_1, \la_2, \la_3, \la_4, \la_5, \la_6 \, | \, \la_7, - \la_7 ), \, g(\la)  \geq 1 \text{ and } f(\la) \geq 8\\
		& (\la_1, \la_2, \la_3, \la_4, \la_5, \la_6 \, | \, \la_7, - \la_7 ), \, g(\la)  = 0, \, \la_1 < \la_2 \text{ and } f(\la) \geq \frac{15}{2}\\
		& (\la_1, \la_2, \la_3, \la_4, \la_5, \la_6 \, | \, \la_7, - \la_7 ), \, g(\la)  = 0, \, \la_1 = \la_2 < \la_3 \text{ and } f(\la) \geq 7\\ 
		& (\la_1, \la_2, \la_3, \la_4, \la_5, \la_6 \, | \, \la_7, - \la_7 ), \, g(\la)  = 0, \, 0 < \la_1 = \la_2 = \la_3 < \la_4 \text{ and } f(\la) \geq \frac{13}{2} \\ 
		& (0, 0, 0, \la_4, \la_5, \la_6 \, | \, \la_7, - \la_7 ), \, g(\la)  = 0, \, 0  < \la_4 \text{ and } f(\la) \geq 6 \\ 
		& (\la_1, \la_2, \la_3, \la_4, \la_5, \la_6 \, | \, \la_7, - \la_7 ), \, g(\la)  = 0, \, 0 < \la_1 = \la_2 = \la_3 = \la_4 < \la_5 \text{ and } f(\la) \geq 6 \\
		& (0, 0, 0, 0, \la_5, \la_6 \, | \, \la_7, - \la_7 ), \, g(\la)  = 0, \, 0 < \la_5 \text{ and } (f(\la) \geq 6 \text{ or }f(\la) = 4) \\ 
		& (\la_1, \la_2, \la_3, \la_4, \la_5, \la_6 \, | \, \la_7, - \la_7 ), \, g(\la)  = 0, \, 0 < \la_1 = \la_2 = \la_3 = \la_4 = \la_5 \text{ and } f(\la) \geq \frac{11}{2} \\ 
		& (0, 0, 0, 0, 0,  \la_6 \, | \, \la_7, - \la_7 ), \, g(\la)  = 0, \,  (f(\la) \geq 4 \text{ or } f(\la) = 2  \text{ or } f(\la) = 0).
	\end{align*}
	In terms of the corresponding parameter $\Lambda = \lambda + \rho$, $L(\lambda)$ is unitary if and only if $\Lambda$ is of the form
	\begin{align*}
		\text{Case 1: }& (\La_1, \La_2, \La_3, \La_4, \La_5, \La_6 \, | \, \La_7, - \La_7 ), \, g(\La)  \geq 2 \text{ and } f(\La) \geq -\frac{1}{2}\\
		\text{Case 2: }& (\La_1, \La_2, \La_3, \La_4, \La_5, \La_6 \, | \, \La_7, - \La_7 ), \, g(\La)  = 1, \, \La_1 < \La_2 - 1 \text{ and } f(\La) \geq -1\\
		\text{Case 3: }& (\La_1, \La_2, \La_3, \La_4, \La_5, \La_6 \, | \, \La_7, - \La_7 ), \, g(\La)  = 1, \, \La_1 = \La_2 - 1 < \La_3 - 2 \text{ and } f(\La) \geq - \frac{3}{2}\\ 
		\text{Case 4: }& (\La_1, \La_2, \La_3, \La_4, \La_5, \La_6 \, | \, \La_7, - \La_7 ), \, g(\La)  = 1, \\
		& 0 < \La_1 = \La_2 - 1 = \La_3 - 2 < \La_4 - 3 \text{ and } f(\La) \geq -2 \\ 
		\text{Case 5: }& (0, 1, 2, \La_4, \La_5, \La_6 \, | \, \La_7, - \La_7 ), \, g(\La)  = 1, \, 3  < \La_4 \text{ and } f(\La) \geq -\frac{5}{2} \\ 
		\text{Case 6: }& (\La_1, \La_2, \La_3, \La_4, \La_5, \La_6 \, | \, \La_7, - \La_7 ), \, g(\La)  = 1, \\
		& 0 < \La_1 = \La_2 - 1 = \La_3 - 2 = \La_4 - 3 < \La_5 - 4 \text{ and } f(\La) \geq -\frac{5}{2} \\
		\text{Case 7: }& (0, 1, 2, 3, \La_5, \La_6 \, | \, \La_7, - \La_7 ), \, g(\La)  = 1, \, 4 < \La_5 \text{ and } (f(\La) \geq -\frac{5}{2} \text{ or }f(\La) = - \frac{9}{2}) \\ 
		\text{Case 8: }& (\La_1, \La_2, \La_3, \La_4, \La_5, \La_6 \, | \, \La_7, - \La_7 ), \, g(\La)  = 1, \\
		& 0 < \La_1 = \La_2 - 1 = \La_3 - 2 = \La_4 - 3 = \La_5 - 4 \text{ and } f(\La) \geq -3 \\ 
		\text{Case 9: }& (0, 1, 2, 3, 4,  \La_6 \, | \, \La_7, - \La_7 ), \, g(\La)  = 1, \\
		& (f(\La) \geq  - \frac{9}{2} \text{ or } f(\La) =  - \frac{13}{2}  \text{ or } f(\La) =  - \frac{17}{2}).
	\end{align*}
	For each $\frg$-dominant infinitesimal character $\Lad$, the following Python code in SageMath returns a list of unitary and nonunitary infinitesimal characters which are $W_{\frg}$--conjugates of $\Lad$.
	\begin{lstlisting}
sage: def g(t):
....:     x1, x2, x3, x4, x5, x6, x7, x8 = t
....:     return (1/2)*(x1 - x2 - x3 - x4 - x5 - x6 - 2*x7)
sage:  def case_1(t):
....: ....:     x1, x2, x3, x4, x5, x6, x7, x8 = t
....: ....:     return (g(t) >= 2)
sage:  def case_2(t):
....: ....:     x1, x2, x3, x4, x5, x6, x7, x8 = t
....: ....:     return (g(t) == 1) and (x1 < x2 - 1) 
sage:  def case_3(t):
....: ....:     x1, x2, x3, x4, x5, x6, x7, x8 = t
....: ....:     return (g(t) ==1) and (x1 == x2 - 1 < x3 - 2)  
sage:  def case_4(t):
....: ....:     x1, x2, x3, x4, x5, x6, x7, x8 = t
....: ....:     return (g(t) == 1) and (0 < x1 == x2 - 1 == x3 - 2 < x4 - 3)
sage:  def case_5(t):
....: ....:     x1, x2, x3, x4, x5, x6, x7, x8 = t
....: ....:     return (g(t) == 1) and (0 == x1 == x2 - 1 == x3 - 2) and (x4 > 3)
sage:  def case_6(t):
....: ....:     x1, x2, x3, x4, x5, x6, x7, x8 = t
....: ....:     return (g(t) == 1) and (0 < x1 == x2 - 1 == x3 - 2 == x4 - 3 < x5 - 4) 
sage:  def case_7(t):
....: ....:     x1, x2, x3, x4, x5, x6, x7, x8 = t
....: ....:     return (g(t) == 1) and (0 == x1 == x2 - 1 == x3 - 2 == x4 - 3) and (x5 > 4)
sage:  def case_8(t):
....: ....:     x1, x2, x3, x4, x5, x6, x7, x8 = t
....: ....:     return (g(t) == 1) and (0 < x1 == x2 - 1 == x3 - 2 == x4 - 3 == x5 - 4)
sage:  def case_9(t):
....: ....:     x1, x2, x3, x4, x5, x6, x7, x8 = t
....: ....:     return (g(t) == 1) and (0 == x1 == x2 - 1 == x3 - 2 == x4 - 3 == x5 - 4) 
sage: def f(t):
....:     x1, x2, x3, x4, x5, x6, x7, x8 = t
....:     return x7
sage: def k_dom_reg(t):
....:      x1, x2, x3, x4, x5, x6, x7, x8 = t
....:      return (abs(x1) < x2 < x3 < x4 < x5) and (g(t) > 0)
sage: def g_dom(t):
....:     x1, x2, x3, x4, x5, x6, x7, x8 = t
....:     return abs(x1) <= x2 <= x3 <= x4 <= x5<=x6 and (g(t) >= 0)
sage:  def unitary(t):
....:      x1, x2, x3, x4, x5, x6, x7, x8 = t
....:      return ((case_1(t) == true) and (k_dom_reg(t) == true) and (f(t) >= -1/2)) or ((case_2(t) == true) and  (k_dom_reg(t) == true) and (f(t) >= -1)) or ((case_3(t) == true) and (k_dom_reg(t) == true) and (f(t) >= -3/2)) or ((case_4(t) == true) and (k_dom_reg(t) == true) and (f(t) >= -2)) or ((case_5(t) == true) and (k_dom_reg(t) == true) and (f(t) >= -5/2)) or ((case_6(t) == true) and (k_dom_reg(t) == true) and (f(t) >= -5/2)) or ((case_7(t) == true) and (k_dom_reg(t) == true) and ((f(t) >= -5/2) or (f(t) == -9/2))) or ((case_8(t) == true) and (k_dom_reg(t) == true) and (f(t) >= -3)) or ((case_9(t) == true) and (k_dom_reg(t) == true) and ((f(t) >= -9/2) or (f(t) == -13/2) or (f(t) == -17/2)))
sage: def unitary_inf_char_list(t):
....:      x1, x2, x3, x4, x5, x6, x7, x8 = t
....:      v = vector([x1, x2, x3, x4, x5, x6, x7, x8])
....:      unitary_list = []
....:      for alpha in W_1:
....:          if (unitary(alpha.to_matrix() * v) and (k_dom_reg(alpha.to_matrix() * v) == true) and ((alpha.to_matrix() * v) not in unitary_list)):
....:              unitary_list.append(alpha.to_matrix() * v)
....:      if not g_dom(t):
....:          return "The entered parameter is not g-dominant"
....:      else:
....:          return(unitary_list)
sage: def nonunitary_inf_char_list(t):
....:      x1, x2, x3, x4, x5, x6, x7, x8 = t
....:      v = vector([x1, x2, x3, x4, x5, x6, x7, x8])
....:      nonunitary_list = []
....:      for alpha in W_1:
....:          if ((unitary(alpha.to_matrix() * v) == false) and (k_dom_reg(alpha.to_matrix() * v) == true) and ((alpha.to_matrix() * v) not in nonunitary_list)):
....:                           nonunitary_list.append(alpha.to_matrix() * v)
....:      if not g_dom(t):
....:           return "The entered parameter is not g-dominant"
....:      else:
....:           return(nonunitary_list)
	\end{lstlisting}
	\begin{ex}
		We get the following lists if we put $t = \rho = \left (0, 1, 2, 3, 4, 5, -\frac{17}{2}, \frac{17}{2} \right )$.
		\begin{lstlisting}
sage: t = (0, 1, 2, 3, 4, 5, -17/2, 17/2)
sage: unitary_inf_char_list(t)
[(0, 1, 2, 3, 4, 5, -17/2, 17/2),
(0, 1, 2, 3, 9, -8, -9/2, 9/2),
(3/2, 5/2, 7/2, 9/2, 11/2, -21/2, -3, 3),
(1, 2, 3, 4, 6, -11, -5/2, 5/2),
(0, 1, 2, 5, 6, -11, -5/2, 5/2),
(1/2, 3/2, 5/2, 9/2, 11/2, -23/2, -2, 2),
(0, 1, 3, 4, 5, -12, -3/2, 3/2),
(-1/2, 3/2, 5/2, 7/2, 9/2, -25/2, -1, 1),
(0, 1, 2, 3, 4, -13, -1/2, 1/2),
(0, 1, 2, 3, 4, -13, 1/2, -1/2)]
sage: nonunitary_inf_char_list(t)
[(0, 1, 2, 3, 5, 4, -17/2, 17/2),
(0, 1, 2, 4, 5, 3, -17/2, 17/2),
(0, 1, 3, 4, 5, 2, -17/2, 17/2),
(0, 2, 3, 4, 5, 1, -17/2, 17/2),
(-1, 2, 3, 4, 5, 0, -17/2, 17/2),
(1, 2, 3, 4, 5, 0, -17/2, 17/2),
(0, 2, 3, 4, 5, -1, -17/2, 17/2),
(3/2, 5/2, 7/2, 9/2, 11/2, -1/2, -8, 8),
(1/2, 5/2, 7/2, 9/2, 11/2, -3/2, -8, 8),
(0, 1, 3, 4, 5, -2, -17/2, 17/2),
(1/2, 3/2, 7/2, 9/2, 11/2, -5/2, -8, 8),
(0, 1, 2, 4, 5, -3, -17/2, 17/2),
(0, 1, 4, 5, 6, -3, -15/2, 15/2),
(1/2, 3/2, 5/2, 9/2, 11/2, -7/2, -8, 8),
(0, 1, 2, 3, 5, -4, -17/2, 17/2),
(0, 1, 3, 5, 6, -4, -15/2, 15/2),
(1/2, 3/2, 5/2, 7/2, 11/2, -9/2, -8, 8),
(0, 1, 2, 3, 4, -5, -17/2, 17/2),
(-1/2, 3/2, 5/2, 11/2, 13/2, -9/2, -7, 7),
(0, 1, 3, 4, 6, -5, -15/2, 15/2),
(1/2, 3/2, 5/2, 7/2, 9/2, -11/2, -8, 8),
(0, 1, 2, 6, 7, -5, -13/2, 13/2),
(-1/2, 3/2, 5/2, 9/2, 13/2, -11/2, -7, 7),
(0, 1, 3, 4, 5, -6, -15/2, 15/2),
(0, 1, 2, 5, 7, -6, -13/2, 13/2),
(-1, 2, 3, 4, 7, -6, -13/2, 13/2),
(-1/2, 3/2, 5/2, 9/2, 11/2, -13/2, -7, 7),
(-1/2, 3/2, 5/2, 9/2, 15/2, -13/2, -6, 6),
(0, 1, 2, 5, 6, -7, -13/2, 13/2),
(-1, 2, 3, 4, 6, -7, -13/2, 13/2),
(0, 1, 3, 4, 8, -7, -11/2, 11/2),
(-1/2, 3/2, 5/2, 9/2, 13/2, -15/2, -6, 6),
(-3/2, 5/2, 7/2, 9/2, 11/2, -15/2, -6, 6),
(1/2, 3/2, 5/2, 7/2, 17/2, -15/2, -5, 5),
(0, 1, 3, 4, 7, -8, -11/2, 11/2),
(-1, 2, 3, 5, 6, -8, -11/2, 11/2),
(1/2, 3/2, 5/2, 7/2, 15/2, -17/2, -5, 5),
(-1/2, 3/2, 7/2, 9/2, 13/2, -17/2, -5, 5),
(0, 1, 2, 3, 8, -9, -9/2, 9/2),
(0, 2, 3, 4, 7, -9, -9/2, 9/2),
(0, 1, 4, 5, 6, -9, -9/2, 9/2),
(-1/2, 3/2, 5/2, 7/2, 15/2, -19/2, -4, 4),
(1/2, 3/2, 7/2, 9/2, 13/2, -19/2, -4, 4),
(0, 1, 3, 4, 7, -10, -7/2, 7/2),
(1, 2, 3, 5, 6, -10, -7/2, 7/2),
(1/2, 3/2, 5/2, 9/2, 13/2, -21/2, -3, 3)]
		\end{lstlisting}
	\end{ex}

\end{document}